\RequirePackage{fix-cm}
\documentclass{svjour3}                     
\smartqed  
\makeatletter
\def\cl@chapter{\@elt {theorem}}
\makeatother

\usepackage{lipsum}
\usepackage{amsfonts}
\usepackage{graphicx}
\usepackage{epstopdf}
\usepackage{algorithmic}
\usepackage{algorithm}
\usepackage{thmtools}
\usepackage{amsmath}
\usepackage{hyperref}
\usepackage{cleveref}
\Crefname{prop}{Proposition}{Propositions}
\Crefname{thm}{Theorem}{Theorems}
\Crefname{assumption}{Assumption}{Assumptions}
\Crefname{lem}{Lemma}{Lemmas}

\usepackage{csvsimple}
\usepackage{booktabs}
\usepackage{siunitx}
\usepackage{numprint}
\usepackage{amsopn}
\usepackage{amssymb}
\renewenvironment{proof}{{\itshape Proof.}}{\qed}

\ifpdf
  \DeclareGraphicsExtensions{.eps,.pdf,.png,.jpg}
\else
  \DeclareGraphicsExtensions{.eps}
\fi

\usepackage{titlesec}
\titleformat{\paragraph}[runin]
{\normalfont\normalsize\bfseries}{\theparagraph}{1em}{}
\titlespacing*{\paragraph}
{18pt}{0.35ex plus 1ex minus .2ex}{1.ex plus .2ex}

\titleformat{\section}[runin]{\normalfont\normalsize\bfseries}{\thesection.}{1em}{}[.]
\titleformat{\subsection}[runin]{\normalfont\normalsize\bfseries}{\thesubsection.}{1em}{}[.]
\titleformat{\subsubsection}[runin]{\normalfont\normalsize\bfseries}{\thesubsubsection.}{1em}{}[.]
\titleformat{\paragraph}[runin]{\normalfont\normalsize\bfseries}{\theparagraph.}{1em}{}[.]

\usepackage{ifthen}
\newboolean{longversion}
\setboolean{longversion}{true}

\newcommand{\R}{\mathbb{R}}
\newcommand{\N}{\mathbb{N}}

\newcommand{\E}{\mathbb{E}}

\newcommand{\domfi}{X_i}
\newcommand{\ones}{\mathbf{1}}
\newcommand{\zeros}{\mathbf{0}}
\newcommand{\wstar}{w^*}

\newcommand{\vstar}{v^*}
\newcommand{\diam}{D}
\newcommand{\diamC}{D_{\mathcal{C}}}
\newcommand{\AK}{\mathcal{A}_K}
\newcommand{\diamAk}{D_{\AK}}

\newcommand{\facialdistance}{\omega}
\newcommand{\FWLinearRateConstant}{\delta}
\newcommand{\ErrA}{\text{Err}_A}
\newcommand{\Errf}{\text{Err}_f}
\newcommand{\perturb}{\theta}

\newcommand{\conv}{\operatorname{conv}}
\newcommand{\po}{\operatorname{Po}}

\newcommand{\cl}{\operatorname{cl}}
\newcommand{\epi}{\operatorname{epi}}
\newcommand{\dom}{\operatorname{dom}}
\newcommand{\ext}{\operatorname{ext}}

\DeclareMathOperator*{\argmax}{arg\,max}
\DeclareMathOperator*{\argmin}{arg\,min}
\usepackage{mathtools}
\newcommand{\norm}[1]{\left\lVert#1\right\rVert}
\newcommand{\abs}[1]{\left\lvert#1\right\rvert}

\newtheorem{assumption}{Assumption}[section]
\numberwithin{assumption}{section}
\newtheorem{thm}{Theorem}[section]
\numberwithin{thm}{section}
\newtheorem{prop}{Proposition}[section]
\numberwithin{prop}{section}
\newtheorem{lem}{Lemma}[section]
\numberwithin{lem}{section}
\newtheorem{cor}{Corollary}[section]
\numberwithin{cor}{section}

\begin{document}

\title{Frank-Wolfe meets Shapley-Folkman: a systematic approach for solving nonconvex separable problems with linear constraints
}


\titlerunning{Nonconvex optimization via Frank-Wolfe \& Carath\'eodory}        

\author{Benjamin Dubois-Taine \and Alexandre~d'Aspremont 
}


\institute{Benjamin Dubois-Taine \at
              DI ENS, Ecole Normale sup\'erieure, Universit\'e PSL, CNRS, INRIA,
Paris, 75005, France  \\
              \email{benjamin.paul-dubois-taine@inria.fr}           
           \and
           Alexandre d'Aspremont \at
              DI ENS, Ecole Normale sup\'erieure, Universit\'e PSL, CNRS, INRIA,
Paris, 75005, France  \\
              \email{aspremon@ens.fr} 
}

\date{Received: date / Accepted: date}

\maketitle

\begin{abstract}
We consider separable nonconvex optimization problems under affine constraints. For these problems, the Shapley-Folkman theorem provides an upper bound on the duality gap as a function of the nonconvexity of the objective functions, but does not provide a systematic way to construct primal solutions satisfying that bound. In this work, we develop a two-stage approach to do so. The first stage approximates the optimal dual value with a large set of primal feasible solutions. In the second stage, this set is trimmed down to a primal solution by computing (approximate) Carath\'eodory representations. The main computational requirement of our method is tractability of the Fenchel conjugates of the component functions and their (sub)gradients. When the function domains are convex, the method recovers the classical duality gap bounds obtained via Shapley-Folkman. When the function domains are nonconvex, the method also recovers classical duality gap bounds from the literature, based on a more general notion of nonconvexity.
\end{abstract}

\section{Introduction}
\label{intro}

We consider nonconvex separable optimization problems of the form
\begin{align}
\tag{P}
    \begin{split}
    \mbox{minimize} & \ \sum_{i=1}^n f_i(x_i)\\
    \mbox{subject to }& \ \sum_{i=1}^n A_ix_i  \leq b,\\
    &x_i \in \domfi, \ i=1, \dots, n,
    \end{split}
    \label{eq:primal}
\end{align}
in the variables $x_i \in \R^{d_i}$, with $A_i \in \R^{m \times d_i}$, $b \in \R^m$. We make the following assumption on the structure of the functions considered.
\begin{assumption}
\label{ass:assumption1}
    For all $i \in \{1, \dots, n\}$, $f_i$ is proper, lower semicontinuous and has an affine minorant. Moreover $\domfi = \dom f_i$ is compact.
\end{assumption}
Note that we do not assume convexity of either the functions $f_i$ or their domains $\domfi$. We also assume that one can efficiently compute the conjugates of the functions $f_i$ and their gradients (or subgradients). We let $d = \sum_{i=1}^n d_i$ and $A = [A_1 \dots A_n]\in \R^{m \times d}$ so that the constraints can be expressed in the more compact form $A x \leq b$, where $x = [x_1^\top x_2^\top \dots x_n^\top]^\top$.

 \subsection{Motivation}
\label{sec:motivation}
Real-world problems with the structure of~\eqref{eq:primal} arise, for example, in power system management~\cite{bertsekas1983optimal,vujanic2016decomposition}, supply chain optimization~\cite{vujanic2014large}, resource allocation~\cite{udell2016bounding}, network utility problems~\cite{bi2016refined} and feature selection~\cite{askari2024naive}. Let us illustrate one such popular application, commonly referred to as the unit commitment problem~\cite{bertsekas1983optimal}. In this problem, an operator manages a network of $n$ plants producing electricity, and must meet a certain electricity demand (given ahead of time) over $m$ time steps, say $D_1, \dots, D_m$. At each time step $t$, the operator can choose how much power $g_i^t$ is produced by a plant $i$. The constraint on the demand can then be written as
\begin{align}
\label{eq:UC-intro-constraints}
    \sum_{i=1}^n g_i^t \geq D_t, \quad t=1, \dots, m.
\end{align}
For each plant $i$, the production profile over $t$ time steps $g_i = [g_i^1, \dots, g_i^t]$ has a certain associated cost $c_i(g_i)$. The goal of the operator is then to minimize the overall cost, i.e.
\begin{align}
\label{eq:UC-intro-objective}
    \mbox{minimize} \ \sum_{i=1}^n c_i(g_i).
\end{align}
Finally, at each time step, a plant $i$ is either turned off, or turned on and producing between a minimum and a maximum capacity. This makes the domain of the objective functions non-convex. We have thus described a non-convex problem whose objective function~\eqref{eq:UC-intro-objective} is separable and whose constraints~\eqref{eq:UC-intro-constraints} are affine, just as in~\eqref{eq:primal}. We will come back in more details to the unit commitment problem in the numerical section of the paper.

\subsection{Background}
 The Shapley-Folkman theorem states that the Minkowski sum of arbitrary sets gets increasingly close to its convex hull as the number of terms in the sum increases. This was first established by Shapley and Folkman in private communications, and first published in~\cite{starr1969quasi}. This result can be used to show that the duality gap of~\eqref{eq:primal} vanishes when the number $n$ of component functions goes to infinity, while the dimension of the constraints $m$ remains bounded~\cite{aubin1976estimates}. More precisely, the duality gap can be bounded by $m \max_i \rho(f_i)$, where $\rho(f_i)$ is a measure of the nonconvexity of the function $f_i$~\cite{aubin1976estimates,bi2016refined,udell2016bounding}. In cases where the domain of the component functions is not convex, the Shapley-Folkman theorem can also be used to prove similar bounds on the duality gap, this time based on more general measures of nonconvexity than the ones mentioned above~\cite{bertsekas2014constrained,vujanic2016decomposition}. Those duality gap bounds are particularly interesting in many real-world applications where $n >> m$. This is often the case for the unit commitment problem described in \Cref{sec:motivation} where $n$, the number of plants producing electricity, is typically much larger than $m$, the number of time steps at which the demand has to be satisfied.

\subsection{Roadmap of the paper and contributions}
The goal of \Cref{sec:preliminaries} is to provide the reader with all the necessary mathematical tools to present duality gap bounds for problems of the form~\eqref{eq:primal}. We emphasize that those bounds are well-known, see for example~\cite{aubin1976estimates,ekeland1999convex}, and are not part of our contributions. They are however central for the main contribution of the paper. These bounds are based on the Shapley-Folkman theorem, which is itself based on the Carath\'eodory decomposition results. After introducing some important definitions, we therefore recall the Carath\'eodory results, in both the convex and conic cases. 

Our first contribution comes at this stage where we adapt the results of \cite{wirth2024frank} to provide and analyze an algorithm which computes the conic Carath\'eodory representation of a given vector. These results allow us to recall a proof of the Shapley-Folkman theorem. We end \Cref{sec:preliminaries} by presenting different (but similar in flavor) duality gap bounds for~\eqref{eq:primal}, notably the central result from Aubin and Ekeland~\cite{aubin1976estimates}.

The bounds on the duality gap presented in \Cref{sec:preliminaries} are not constructive, in the sense that the proofs do not provide a solution meeting those bounds. The main contribution of this work is a systematic approach to obtain solutions of~\eqref{eq:primal} meeting the Shapley Folkman duality gap bounds. We detail this approach in~\Cref{sec:nonconvex-optimization}, and show that it can be broken down into two main stages. 

The first stage, described in~\cref{sec:nonconvex-obtaining-primal-points}, is a conditional gradient method on a smooth convex objective aiming at approximating the optimal dual value. This method converges sublinearly and, crucially here, we show that the linear minimization oracle comes down to computing the Fenchel conjugate of the functions $f_i$, a fairly routine step. The result of the algorithm can be cast as a conic combination of extreme points of the primal feasible set. 

The second stage of the method, described in~\Cref{sec:nonconvex-trimming}, requires trimming this conic combination to a conic Carath\'eodory representation using only $n+m+1$ elements. This can be done using the algorithm derived in \Cref{sec:preliminaries}, justifying our first contribution. When the domains of the functions $\domfi$ are convex, we show in~\Cref{sec:final-reconstruction} that this directly yields a primal feasible solution meeting the classical bounds of~\cite{aubin1976estimates,ekeland1999convex} presented in \Cref{sec:preliminaries}, up to an arbitrarily small additive term. This is summarized in Theorem~\ref{thm:convergence-convex-domain}, which we consider to be the main result of this work. In the case of nonconvex domains, we show that the obtained solution is not feasible on at most $m+1$ of the $n$ component functions. Under various additional assumptions, we show that a final reconstruction step can be carried out to obtain fully feasible solutions, and we derive duality gap bounds similar to the ones above, this time based on a more general notion of nonconvexity~\cite{bertsekas2014constrained,vujanic2014large,vujanic2016decomposition}. The main results for nonconvex domains are stated in Theorem~\ref{thm:convergence-noncconvex-domain1} and Theorem~\ref{thm:convergence-noncconvex-domain2}.

Let us emphasize that the main requirements of the method described in \Cref{sec:nonconvex-optimization} are (i) the computation of the optimal dual value of~\eqref{eq:primal} (ii) being able to compute Fenchel conjugates of the functions $f_i$. These fairly routine computational assumptions make our method versatile and applicable in many real-world settings.

In \Cref{sec:extension-approx-carath}, we discuss a variant of our method described in \Cref{sec:nonconvex-optimization} to solve problem~\eqref{eq:primal}. As mentioned earlier, the second stage of our method requires computing a Carath\'eodory representation of the output of the first stage, and this can be done using the algorithm presented in \Cref{sec:preliminaries}. However, this algorithm has high complexity and can therefore be quite slow on large-scale instances of~\eqref{eq:primal}. Inspired by~\cite{combettes2023revisiting}, we show in \Cref{sec:extension-approx-carath} that one can use conditional gradient methods to obtain an approximate Carath\'eodory representation of the output of the first stage instead of an exact one. Those methods have a fast linear rate of convergence. Moreover, we show that using such a conditional gradient method also reduces the dependency on $m$ in the duality gap bounds of the reconstructed solutions, at the price of an exponentially decreasing error term, which depends on the number of iterations of the conditional gradient method.

Finally, in \Cref{sec:numerical-results}, we implement and test our method on two numerical applications, comparing final performances and running times of computing exact versus approximate Carath\'eodory representations.

\subsection{Related work}

When the underlying domain of the functions are convex,~\cite{udell2016bounding} describes a randomized algorithm achieving the bounds derived from Shapley-Folkman with probability one. To the best of our knowledge, this is the only other work describing a method to build solutions matching these bounds. However, the approach requires knowledge of the full solution space of the convex relaxation of~\eqref{eq:primal}, as opposed to our method which only requires access to the optimal dual value. We give a further comparison of our method with the method of~\cite{udell2016bounding} in \Cref{sec:duality-gap-estimation}, where the tools necessary for a further discussion are introduced.

In the special case of mixed-integer linear programs (MILPs) with unique solution,~\cite{vujanic2014large,vujanic2016decomposition} describe algorithms achieving the bounds for nonconvex domains derived from Shapley-Folkman under various additional assumptions on the domain.

Let us also mention here that the bounds derived from Shapley-Folkman can be slightly improved. In~\cite{bi2016refined}, the authors prove duality gap bounds based on a slightly more refined version of the nonconvexity of a function, known as $k$-th nonconvexity. Although our results immediately extend to this more refined version, we omit them to simplify the exposition.

In~\cite{kerdreux2017approximate}, the authors describe an approach to obtain approximate \newline Carath\'eodory representations with high sampling ratios, which they use to prove approximate Shapley-Folkman theorems. Those, in turn, can be used to prove duality gap bounds which reduce the dependence on the constraint dimension $m$, at the price of an additional term due to the approximation error. The approach is, however, not constructive, and the duality gap bounds provided are on a perturbed version of the primal problem.

Finally, let us briefly mention branch and bound algorithms. While a complete overview is beyond the scope of this work (we refer the reader to~\cite{lawler1966branch,horst2013global,tawarmalani2013convexification} for a thorough exposition), these algorithms form a widely used family of methods for tackling nonconvex problems. We therefore highlight the main differences between these methods and our approach. Branch and bound algorithms aim to find exact solutions of the problem by recursively partitioning the feasible set and computing upper and lower bounds on the optimal function value within each partition. In particular, the design of the partitioning step is a key aspect of these methods, and must often be tailored to the problem at hand. Moreover, computing the bounds also depends on the structure of the problem, and the number of partitions required to obtain a solution can be prohibitively large in high dimension. On the other hand, our method is aimed at finding good approximate solutions in reasonable time, and relies on fairly general assumptions (computation of Fenchel conjugates), making it applicable without further specification in a variety of settings.

\section{Preliminaries} 
\label{sec:preliminaries}
For $k\in \N$, $\Delta_k$ is the simplex of dimension $k-1$. For a set $S$, we write its closure $\cl(S)$, its convex hull $\conv (S)$ and the set of its extreme points $\ext(S)$. For a vector $x \in \R^d$, we define $\norm{x}_+ = \sqrt{\sum_{j=1}^d \max \left( x_j, 0 \right)^2}$.

If $f$ is a function not everywhere $+ \infty$, with an affine lower bound, the (Fenchel) conjugate of $f$ is defined as
\begin{align}
    f^*(y) = \sup_{x \in \dom f} \{ y^\top x - f(x) \}.
\end{align}
The biconjugate of $f$, defined as the conjugate of $f^*$ and written $f^{**}$, is also the pointwise supremum of all affine functions minorizing $f$~\cite[Theorem X.1.3.5]{hiriart1996convex}. In particular, we have $\epi f^{**} = \cl \conv \epi f$. Moreover, when $f$ is lower semicontinuous and 1-coercive, $\conv \epi f$ is closed and hence $\epi f^{**} = \conv \epi f$~\cite[Theorem X.1.5.3]{hiriart1996convex}.

The nonconvexity of a function $f: \R^d \rightarrow \R$ is defined as~\cite{aubin1976estimates}
\begin{align}
    \label{eq:def-nonconvexity-measure}
    \rho(f) = \sup \left\{ f\left(\sum_{j=1}^p \alpha_j x_j \right) - \sum_{j=1}^p \alpha_j f(x_j) \mid p\in \N, x_j \in \dom f, \alpha \in \Delta_p \right\}.
\end{align}
Clearly $\rho(f) = 0$ if $f$ is convex. Note also that in the above definition, we do not require $\sum_{j=1}^p \alpha_j x_j$ to be in $\dom f$. In particular we then have $\rho(f) = \infty$ if $\dom f$ is not convex. Under Assumption~\ref{ass:assumption1}, the nonconvexity $\rho(f)$ can also be expressed in terms of the biconjugate of $f$~\cite[Proposition 6]{bi2016refined}, 
\begin{align*}
    \rho(f) = \sup_{x\in \dom f^{**}} \{f(x) - f^{**}(x)\}.
\end{align*}

\subsection{Dual and convex relaxation}
Following~\cite{kerdreux2017approximate} we define, for $\lambda \in \R^m$, 
\begin{align}
    \label{eq:def-Psi}
    \Psi(\lambda) = \inf_{x_i \in \domfi} \left\{\sum_{i=1}^n f_i(x_i) + \lambda^\top (\sum_{i=1}^n A_i x_i - b)\right\}
\end{align}
The Lagrange dual of~\eqref{eq:primal} is given by
\begin{align}
\tag{D}
    \begin{split}
    \mbox{maximize} & \ \Psi(\lambda) \\
    \text{subject to }& \ \lambda \geq 0.
    \end{split}
    \label{eq:dual}
\end{align}
The bidual of~\eqref{eq:primal}, which is the dual of~\eqref{eq:dual}, is then given by
\begin{align}
\tag{CoP}
\begin{split}
\mbox{minimize} & \ \sum_{i=1}^n f_i^{**}(x_i)\\
\mbox{subject to }& \ Ax  \leq b,\\
&\ x_i \in \conv (\domfi), \quad i=1, \dots, n.
\end{split}
\label{eq:bi-dual}
\end{align}
The fact that the domain is now $\conv (\domfi)$ follows from $\dom f_i^{**} = \conv (\dom f_i)$. From weak duality, the optimal value of~\eqref{eq:dual} is a lower bound on the optimal value of~\eqref{eq:primal}. The following proposition from~\cite{lemarechal2001geometric} gives a sufficient condition for the optimal values of~\eqref{eq:dual} and~\eqref{eq:bi-dual} to match.
\begin{prop}
    Assume $\Psi$ is not everywhere $- \infty$ and there exists a feasible $\hat{x}$ in the interior of $\dom \sum_{i=1}^n f_i^{**}$. Then strong duality holds, i.e. the optimal values of~\eqref{eq:dual} and~\eqref{eq:bi-dual} match.
\end{prop}
From now on, we assume that strong duality holds and let
\begin{align}
\label{eq:def_v_star}
    \vstar := \max_{\lambda \geq 0} \Psi(\lambda).
    \end{align}
    In the rest of the paper, we assume that one can efficiently obtain $\vstar$, by either solving~\eqref{eq:dual} or~\eqref{eq:bi-dual}. This is a relatively fair assumption, as both~\eqref{eq:dual} and~\eqref{eq:bi-dual} are convex problems, for which standard convex optimization solvers can be used.

\subsection{Carath\'eodory}
Carath\'eodory's theorem and its conic version are core ingredients in proving the Shapley-Folkman theorem and the resulting duality gap bounds. We recall both results below, and also present the proof of the conic version, as it will be useful later in the paper. The classical Carath\'eodory theorem states that a vector in the convex hull of a set in $\R^{p}$ can be written as a convex combination of at most $p+1$ points of the original set. 

\begin{thm}[Carath\'eodory]
\label{thm:caratheodory}
    Let $S \subset \R^p$ and $w \in \conv S$. Then there exists $\{\lambda_j\}_{j=1}^{p+1} \subset \R_+$, $\{w_j\}_{j=1}^{p+1} \subset S$ with $\sum_{j=1}^{p+1} \lambda_j = 1$, such that $w = \sum_{j=1}^{p+1} \lambda_j w_j$.
\end{thm}
\begin{proof}
    See~\cite[Theorem 17.1]{rockafellar2015convex}.
\end{proof}

Let $\po (S) = \left\{ \sum_{j} \lambda_j v_j \mid \lambda_j \geq 0, v_j \in S\right\}$ be the conic hull of $S$. The conic Carath\'eodory theorem states that any point in the conic hull can be written as a conic combination of at most $p$ points in the original set.

\begin{restatable}[Conic Carath\'eodory]{thm}{ConicCaratheodory}
\label{thm:conic-caratheodory}
    Let $S \subset \R^p$ and $w \in \po S$. Then there exists $\{\lambda_j\}_{j=1}^{p} \subset \R_+$, $\{w_j\}_{j=1}^{p} \subset S$, such that $w = \sum_{j=1}^{p} \lambda_j w_j$.
\end{restatable}
 \begin{proof}
    Since $w \in \po(S)$, there exists $N \in \N$, $\lambda \in \R^N_+ $ and $w_1, \dots, w_N \in S$ such that $w = \sum_{j=1}^N \lambda_j w_j$. We assume $N > p$ (if $N\leq p$, we are already done). We claim that we can rearrange the coefficients so that one of the $w_j$'s has coefficient $0$. To do so, suppose $\lambda_j > 0$ for all $j=1, \dots, N$ (otherwise we are done). Since $N > p$, the $w_j$'s are linearly dependent, so there exists $\delta_1, \dots, \delta_N$ not all zero such that
    \begin{align*}
        \sum_{j=1}^N \delta_j w_j = 0.
    \end{align*}
    Let $t^* = \max \{ t \geq 0 \mid \lambda_j - t \delta_j \geq 0 \text{ for all $j = 1,\dots, N$}\}$ and set $\lambda_j' = \lambda_j - t^* \delta_j$ for all $j$. Then $\lambda_j' \geq 0$ for all $j$. Moreover, there exists $j_0$ such that $\lambda_{j_0}' = 0$ and
    \begin{align*}
        \sum_{j=1}^N \lambda_j'w_j &= \sum_{j=1}^N \lambda_j w_j - t^* \sum_{j=1}^N \delta_jw_j = w,
    \end{align*}
    which proves that we can indeed reduce the number of strictly positive coefficients to $N-1$. This argument can be repeated for $r$ iterations as long as $N - r > p$, which proves the theorem.
\end{proof}

The remainder of this subsection is dedicated to the design of an algorithm for computing conic Carath\'eodory representations. We consider a vector $\wstar \in \R^p$, and a subset $S \subset \R^p$ of size $N$ which we write $S= \{w_1, \dots, w_N\}$, and we assume that $\wstar \in \po(S)$. In other words, for coefficients $\lambda_1, \dots, \lambda_N > 0$, we have
\begin{align}
    \label{eq:w_definition}
    \wstar = \sum_{j=1}^N \lambda_j w_j.
\end{align}

We assume that the vectors $w_1, \dots, w_N$ and the coefficients $\lambda_1, \dots, \lambda_N$ are known.  Let us give some intuition behind the algorithm to reduce the number of components from $N$ to $p$, which is also the main idea behind the proof of Theorem~\ref{thm:conic-caratheodory}. Since $N > p$, the vectors $w_1, \dots, w_N$ are linearly dependent. Finding the coefficients for which the resulting linear combination is 0 entails solving a linear system in dimension $\R^{p \times N}$. Since this linear system has a trivial solution (namely, setting all coefficients to 0), an additional random row ensures that non-trivial solutions are obtained (with probability one). 

Once the non-trivial solution $\delta$ has been found, a simple algebraic trick shows that one can subtract the (scaled) solution $\delta$ from the original $\lambda$ conic coefficients, and that the resulting conic coefficients have at least one coordinate equal to 0, which we can thus remove from the conic representation. Iterating this procedure will yield a conic representation with precisely $p$ elements. Now, instead of considering all of the vectors $w_1, \dots, w_N$ in the first step above, it is enough to only consider $p+1$ of them, as those are also linearly dependent, and this allows to reduce the dimension of the linear system to be solved at each iteration to $\R^{(p+1) \times (p+1)}$. We summarize the method in \Cref{alg:Constructive-caratheodory}. For a matrix $W$, we write $W_{:, k}$ as its $k$-th column and $W_{k, :}$ as its $k$-th row. Moreover, $W_{:, k_1:k_2}$ is the submatrix of $W$ composed of the $k_2 - k_1 + 1$ columns between $k_1$ and~$k_2$.

\begin{algorithm}
\caption{Exact Carath\'eodory}
\label{alg:Constructive-caratheodory}
\begin{algorithmic}[1]
\STATE{\textbf{Input}: $w_1, \dots, w_N$ and nonnegative coefficients $\lambda_1, \dots, \lambda_N$ such that $\wstar = \sum_{j=1}^N \lambda_j w_j$.}
\STATE{Build $W \in \R^{p \times N}$, the matrix whose columns are the vectors $w_j$.}
 \STATE{Draw vector $u \in \R^{N}$ with uniform distribution on the unit sphere and let $\tilde{W}_u = \begin{pmatrix}
    W \\ u^\top
\end{pmatrix}$. }
\STATE{$\tilde{W}^s_u \leftarrow (\tilde{W}_u)_{:,1: p+1}$}
\STATE{$\alpha \leftarrow \lambda_{1:p+1}$}
\STATE{$c \leftarrow p+ 2 $}
\WHILE{$c \leq N + 1$}
\STATE{Solve $(\tilde{W}^s_u) \delta = \begin{pmatrix}
    \zeros_{p} \\ 1
\end{pmatrix}$}
\STATE{$t^* \leftarrow \min \{\frac{\alpha_j}{\delta_j} \mid \delta_j > 0\}$}
\STATE{$\alpha' \leftarrow \alpha - t^* \delta$}
\STATE{$j^* \leftarrow \argmin_j \alpha_j'$}
\STATE{Remove element at index $j^*$ from $\alpha'$ and column $j^*$ from $\tilde{W}^s_u$}
\IF{$c = N + 1$}
\STATE{$\alpha \leftarrow \alpha'$}
\STATE{\textbf{Break}}
\ELSE
\STATE{$\alpha \leftarrow (\alpha', \lambda_{c})$}
\STATE{$\tilde{W}^s_u \leftarrow \left(\tilde{W}^s_u , (\tilde{W}_u)_{:, c}\right)$}
\STATE{$c \leftarrow c + 1$}
\ENDIF
\ENDWHILE
\STATE{Return $\alpha, (\tilde{W}^s_u)_{1:p,:}$.}
\end{algorithmic}
\end{algorithm}

\begin{thm}
\label{thm:Constructive-caratheodory}
With probability 1, \Cref{alg:Constructive-caratheodory} outputs $\alpha \in \R^{p}_+$ and $p$ vectors $\{w_{j_l}\}_{l=1}^{p} \subset S$, such that
\begin{align}
\label{eq:constructive-caratheodory-output}
    \wstar  = \sum_{l=1}^{p}\alpha_l
        w_{j_l}.
\end{align}
\end{thm}
\begin{proof}
Our proof basically mirrors the one of Theorem~\ref{thm:conic-caratheodory}. By the definition of $W$ in \Cref{alg:Constructive-caratheodory}, we have $W \lambda = \wstar$. Let $W^s$ be the matrix $\tilde{W}^s_u$ from which we have removed the last row (namely the row composed of elements of the random vector $u$). At the beginning of the first iteration, we then have
\begin{align}
    &\ W\lambda = W^s \alpha + (W_{:, c:N})\lambda_{c:N}, \nonumber\\
\label{eq:proof-exact-caratheodory-equality}
    \Rightarrow & \ W^s \alpha = \wstar - (W_{:, c:N})\lambda_{c:N}.
\end{align}
We claim that~\eqref{eq:proof-exact-caratheodory-equality} is true at the beginning of every iteration. We saw that it was true at the beginning of the first iteration, now suppose it is true at the beginning of some iteration.

Since $W^s \in \R^{p \times (p+1)}$, the system $W^s \delta = \zeros_p$ must have a non-trivial solution. Let $\delta^*$ be one such solution. Observe that
\begin{align*}
    \tilde{W}^s_u \frac{\delta^*}{((\tilde{W}^s_u)_{p+1, :})^\top \delta^*} = \begin{pmatrix}
    \zeros_{p} \\ 1
\end{pmatrix},
\end{align*}
and hence a solution of the system on line 8 exists as long as $((\tilde{W}^s_u)_{p+1, :})^\top \delta^* \not = 0$. Since the last row $\tilde{W}^s_u$ has uniform distribution on the unit sphere, this holds with probability 1.

One can then check that lines 9-11 in the algorithm are simple computations to find $\alpha' \geq 0$ such that $W^s \alpha' =W^s \alpha$ and $\alpha_{j^*}' = 0$. Removing the index $j^*$ from $\alpha'$ and the column $j^*$ from $W^s$, we get $\alpha' \in \R^{p}$, $W^{s} \in \R^{p \times p}$ and
\begin{align*}
    W^s \alpha' = \wstar - (W_{:, c:N})\lambda_{c:N}.
\end{align*}
At this point two cases arise. 

\textbf{Case 1}: If $c = N + 1$, it must be that $W^s \alpha' = \wstar$, and we have therefore found a conic combination of $\wstar$ with $p$ coefficients, so we are done. 

\textbf{Case 2}: If $c < N + 1$, we can append $\lambda_c$ to $\alpha'$ and column $c$ of $\tilde{W}$ to $\tilde{W}^s$, as on lines 17 and 18. This results in $\alpha \in \R^{p+1}$ and $W^s \in \R^{p \times (p+1)}$ such that
\begin{align*}
    W^s \alpha = \wstar - (W_{:, c+1:N})\lambda_{c+1:N}.
\end{align*}
Updating $c \leftarrow c+ 1$ as on line 19, and starting a new iteration, we are back to the equation~\eqref{eq:proof-exact-caratheodory-equality}, except that we have reduced the column dimension of the matrix $W_{:, c:N}$ by $1$. We can iterate this argument until case 1 above is true.
\end{proof}

Note that similar algorithms to obtain Carath\'eodory representations have been developed in~\cite{beck2017linearly} and in~\cite{wirth2024frank}. Those approaches are however aimed at reducing the size of the active set of conditional gradient methods, and hence are looking for convex representations via the classic Carath\'eodory result (Theorem~\ref{thm:caratheodory}). In contrast, we are looking for representations in the conic hull of $S$ via the proof of Theorem~\ref{thm:conic-caratheodory}. Our approach is closer to the one of~\cite{wirth2024frank}, in the sense that we make no assumption on how the linear system is solved on line 8. In particular, when the vectors in $S$ are sparse, efficient direct or indirect sparse solvers may be used. In contrast, the approach in~\cite{beck2017linearly} requires maintaining a matrix in row-echelon form which, depending on the application, might not be computationally efficient. Without further assumption on the structure of the linear systems, we suggest a strategy to achieve complexity $O(p^2)$ at each iteration, which is similar to the complexity achieved in~\cite{beck2017linearly} and in~\cite{wirth2024frank}.

First, note that at each iteration, the linear system to be solved is the same as the one at the previous iteration, up to one column deletion and one column insertion. One can therefore compute the QR decomposition of $\tilde{W}^s_u$ on line 4, solve the linear systems on line 8 using the decomposition, and update the QR decomposition by either deleting a column (line 12) or by appending a column (line 18). We give the complexity of this strategy in the next lemma.

\begin{lem}
    \label{lem:exact-carath-flop-count}
    Using the above strategy, the complexity of \Cref{alg:Constructive-caratheodory} is $O(Np^2)$.
\end{lem}
\begin{proof}
    The complexity of computing the first QR decomposition is $O(p^3)$. The complexity of solving the linear system given the QR decomposition is $O(p^2)$, while the complexity of deleting/appending a column is also $O(p^2)$~\cite[Section 6.5.2]{golub2013matrix}. We must do the above two tasks at each iteration, and there are $N - p$ iterations in total. Since $N \geq p$, the complexity is indeed $O(N p^2)$.
\end{proof}

\subsection{The Shapley Folkman theorem}
Both versions of Cara\-th\'eodory's theorem are central to the proof of the Shapley Folkman theorem, which states that the Minkowski sum of arbitrary sets gets increasingly closer to its convex hull as the number of terms grows. Established in private communications between Shapley and Folkman, it was first published in~\cite{starr1969quasi}.
\begin{thm}[Shapley Folkman]
Let $V_{i} \subset \R^d$, $i=1, \dots, n$ and \\$x \in \sum_{i=1}^n \conv (V_i)$. Then for some set $S \subset \{1, \dots, n\}$ with $|S| \leq d$, we have
\begin{align*}
    x \in \sum_{i \in \{1, \dots, n\} \setminus S } V_i + \sum_{i \in S} \conv (V_i)
\end{align*}
\end{thm}
\begin{proof}
Since $x \in \sum_{i=1}^n \conv(V_i)$, by classic Carath\'edory we have
    \begin{align*}
        x = \sum_{i=1}^n \sum_{j=1}^{d+1} \lambda_{ij} v_{ij}
    \end{align*}
    with $v_{ij } \in V_i$ for all $i$, and $\sum_{j=1}^{d+1} \lambda_{ij} = 1$ for all $i$. We can then write
    \begin{align*}
        \begin{pmatrix}
            x \\ \ones_n 
        \end{pmatrix} &= \sum_{i=1}^n \sum_{j=1}^{d+1} \lambda_{ij} \begin{pmatrix}
            v_{ij} \\ e_i
        \end{pmatrix}.
    \end{align*}
    This is a conic combination in $\R^{n+d}$, so by conic Carath\'eodory, there exists $\mu \in \R^{n(d+1)}_+$ with at most $n+d$ nonzero such that
    \begin{align*}
        \begin{pmatrix}
            x \\ \ones_n 
        \end{pmatrix} &= \sum_{i=1}^n \sum_{j=1}^{d+1} \mu_{ij} \begin{pmatrix}
            v_{ij} \\ e_i
        \end{pmatrix}.
    \end{align*}
    In particular, for each $i$, there must be at least one of the $\mu_{ij}$ not equal to $0$. Since there are only $n+d$ nonzero coefficients, this implies that for at least $n-d$ indices, the convex combination is trivial, i.e. has only one nonzero coefficient, which is equal to one. For the remaining $d$ indices, the convex combination is non-trivial. In other words, for a set $S$ with $|S| \leq d$,
    \begin{align*}
        x \in \sum_{i \not \in S} V_i + \sum_{i \in S} \conv(V_i).
    \end{align*}
\end{proof}
\subsection{Duality gap estimation}
\label{sec:duality-gap-estimation}
The Shapley-Folkman theorem established above is one of the main ingredients for proving duality gap bounds for problem~\eqref{eq:primal}. Let us first state an important lemma, which will be useful throughout this section.
\begin{lem}
\label{lem:shapley-folkman-applied-nonconvex-duality}
    Let $\vstar$ be the dual optimal value as defined in~\eqref{eq:def_v_star}. There exists a set $S \subset \{1, \dots, n\}$ of cardinality at most $m+1$ and a vector $\bar{x} \in \R^d$, such that
    \begin{align*}
        &\bar{x}_i \in \domfi \ \text{ for } i \not \in S,\\
        &\bar{x}_i \in \conv \domfi \ \text{ for } i \in S,\\
        &A \bar{x} \leq b,
    \end{align*}
    and
    \begin{align*}
   \vstar = \sum_{i\in S} f_i^{**}(\bar{x}_i) + \sum_{i\not \in S} f_i(\bar{x}_i)
\end{align*}
\end{lem}
\begin{proof}
    Since $\vstar$ is the optimal value of~\eqref{eq:bi-dual}, there exists $\tilde{x} \in \R^d$ with $\tilde{x}_i \in \conv \domfi$ for all $i=1, \dots n$, $A \tilde{x} \leq b$, and $\vstar = \sum_{i=1}^n f_i^{**}(\tilde{x}_i)$. Let
    \begin{align*}
        \tilde{z} = \begin{pmatrix}
            \sum_{i=1}^n f_i^{**}(\tilde{x}_i) \\
            b
        \end{pmatrix}
    \end{align*}
    we have
    \begin{align*}
        \tilde{z} &\in \sum_{i=1}^n \left\{ \begin{pmatrix}
            f_i^{**}(x_i) \\ A_i x_i 
        \end{pmatrix} \mid x_i \in \conv \domfi \right\} + \R^{m+1}_+ \\
        &= \sum_{i=1}^n \conv \left\{ \begin{pmatrix}
            f_i(x_i) \\ A_i x_i 
        \end{pmatrix} \mid x_i \in \domfi \right\} + \R^{m+1}_+ \quad \text{ (Lemma~\ref{lem:equality-epigraph})}.
    \end{align*}
    We can now apply the Shapley-Folkman theorem to show that there exists a set $S$ of cardinality at most $m+1$ such that
    \begin{align*}
        \tilde{z} &\in \sum_{i\not \in S} \left\{ \begin{pmatrix}
            f_i(x_i) \\ A_i x_i 
        \end{pmatrix} \mid x_i \in \domfi \right\} + \sum_{i\in S} \conv \left\{ \begin{pmatrix}
            f_i(x_i) \\ A_i x_i 
        \end{pmatrix} \mid x_i \in \domfi \right\} + \R^{m+1}_+ \\
        &= \sum_{i\not \in S} \left\{ \begin{pmatrix}
            f_i(x_i) \\ A_i x_i 
        \end{pmatrix} \mid x_i \in \domfi \right\} + \sum_{i\in S} \left\{ \begin{pmatrix}
            f_i^{**}(x_i) \\ A_i x_i 
        \end{pmatrix} \mid x_i \in \conv \domfi \right\} + \R^{m+1}_+,
    \end{align*}
    where we applied Lemma~\ref{lem:equality-epigraph} again. In other words, there exists $\bar{x} \in \R^d$, $\bar{x}_i \in \domfi$ for $i \not \in S$ and $\bar{x}_i \in \conv \domfi$ for $i \in S$, and $\delta_1 \geq0, \delta_2 \in \R^m_+$ such that
    \begin{align*}
        \vstar = \sum_{i=1}^n f_i^{**}(\tilde{x}_i) &= \sum_{i\not \in S} f_i(\bar{x}_i) + \sum_{i \in S} f_i^{**}(\bar{x}_i) + \delta_1,\\
        b &= A \bar{x} + \delta_2.
    \end{align*}
    In particular, $A \bar{x} \leq b$, and since $f_i^{**} \leq f_i$, we must actually have $\delta_1 = 0$ and 
    \begin{align*}
        \sum_{i=1}^n f_i^{**}(\tilde{x}_i) = \sum_{i\not \in S} f_i(\bar{x}_i) + \sum_{i \in S} f_i^{**}(\bar{x}_i).
    \end{align*}
\end{proof}

\subsubsection{Convex domain, nonconvex objectives}
When dealing with convex domains $\domfi$, we have $\domfi = \conv(\domfi)$. We now review some results covering duality gap estimation. We first recall the original result of~\cite{aubin1976estimates} (see also~\cite{ekeland1999convex}), giving a bound on the duality gap which depends on the dimension of the constraints $m$ and the nonconvexity of the considered functions, but which is independent of the number of functions $n$.

\begin{restatable}{thm}{NonConvexDualityGap}
\label{thm:nonconvex-duality-gap}
    Let $\vstar$ be the dual optimal value defined in~\eqref{eq:def_v_star} and $\rho(f_i)$ be as defined in~\eqref{eq:def-nonconvexity-measure}. Assume that $\domfi$ is convex for all $i$. There exists $x^* \in \R^d$ such that $A x^* \leq b$ and 
    \begin{align*}
   \vstar \leq  \sum_{i=1}^n f_i(x_i^*) \leq \vstar + (m+1) \max_i \rho(f_i).
\end{align*}
\end{restatable}
 \begin{proof}
    By Lemma~\ref{lem:shapley-folkman-applied-nonconvex-duality}, there exists $\bar{x} \in \R^d$ and set $S \subset \{1, \dots, n\}$ of cardinality at most $m+1$ such that $A\bar{x} \leq b$, $\bar{x}_i \in \domfi$ for $i \not \in S$ and $\bar{x}_i \in \conv \domfi = \domfi$ (since $\domfi$ is convex by assumption) such that
    \begin{align*}
        \vstar = \sum_{i \not \in S} f_i(\bar{x}_i) + \sum_{i \in S} f_i^{**}(\bar{x}_i)
    \end{align*}
    Setting $x^* = \bar{x}$, observe that $x^*$ is primal feasible. Moreover,
    \begin{align*}
        \vstar &\leq \sum_{i=1}^n f_i(x_i^*) \\
        &= \sum_{i\not \in S} f_i(x_i^{*}) + \sum_{i \in S} f_i(x_i^*) \\
        &= \vstar + \sum_{i \in S} (f_i(x_i^*)  - f_i^{**}(x_i^*))\\
        &\leq \vstar + (m+1)\max_i\rho(f_i).
    \end{align*}
\end{proof}
 The above theorem is an existence result, and the proof does not provide a way to construct a solution satisfying the bound.

In~\cite{udell2016bounding}, the authors provide an algorithm achieving the bound of Theorem~\ref{thm:nonconvex-duality-gap} with probability 1, by solving a randomized version of the convex relaxation~\eqref{eq:bi-dual}. Although they claim otherwise, their analysis implicitly assumes that the function domains $\domfi$ are convex. Indeed, their randomized approach yields a solution $x_i \in \conv \domfi$, $i=1, \dots, n$. On a certain index set $\mathcal{J}$ of cardinality at least $n-m$, they show that $i \in \mathcal{J}$ implies that $x_i$ is an extreme point of $\conv \domfi$. This in turn implies that $x_i \in \domfi$ and $f_i(x_i) = f_i^{**}(x_i)$. On non-extreme points, i.e. for $i \not \in \mathcal{J}$, the authors bound $f_i(x_i) - f_i^{**}(x_i)$ by $\rho(f_i)$. However, this implicitly assumes $\domfi$ is convex, since otherwise $\rho(f_i) = \infty$ and the bound is meaningless. When $\domfi$ is nonconvex, an additional step is required to `primalize' the solutions for $i \not \in \mathcal{J}$. We explore such primalization steps in \Cref{sec:final-reconstruction}. Moreover, the approach described in~\cite{udell2016bounding} requires optimizing over the solution set of the convex relaxation~\eqref{eq:bi-dual}, which in some cases might be a difficult task. In contrast, the approach we present in \Cref{sec:nonconvex-optimization} only assumes access to the value of $\vstar$, which can be obtained by either solving the dual~\eqref{eq:dual} or the convex relaxation~\eqref{eq:bi-dual}, whichever one is easier to optimize.


\subsubsection{General domains} When dealing with nonconvex domains, we have $\dom f_i^{**} = \conv \domfi \not = \domfi$. In particular, $\rho(f_i) = \infty$, and thus a more general notion of nonconvexity is needed. As in~\cite{bertsekas2014constrained}, let us define
\begin{align}
    \label{eq:def-max-fi}
    \gamma(f_i) := \sup_{x\in \domfi} f_i(x) - \inf_{y \in \domfi} f_i(y).
\end{align}
 In order to prove a duality gap bound similar to the one of Theorem~\ref{thm:nonconvex-domain-duality-gap}, an additional assumption on the domain of the functions is necessary.

\begin{restatable}{assumption}
{AssNonConvexDomain}
    \label{ass:nonconvex-domain1}
    For all $i=1, \dots, n$ and all $x_i \in \conv \domfi$, there exists $\hat{x}_i \in \domfi$ such that $A_i \hat{x}_i \leq A_i x_i$.
\end{restatable}
We will discuss in \Cref{rem:comments-on-assumption} the relevance of this assumption in real-world applications. For now, let us recall and prove the following theorem, which states that the duality gap can be bounded by a term depending on the number of constraints $m$ and on $\max_i \gamma(f_i)$, but not on the number of functions $n$~\cite{aubin1976estimates,bertsekas2014constrained}.

\begin{restatable}{thm}{NonConvexDomainDualityGap}
\label{thm:nonconvex-domain-duality-gap}
    Let $\vstar$ be the dual optimal value defined in~\eqref{eq:def_v_star} and $\gamma(f_i)$ be defined in~\eqref{eq:def-max-fi}. Suppose that Assumption~\ref{ass:nonconvex-domain1} holds. Then there exists $x^* \in \R^d$ such that $x^*_i \in \domfi$ for all $i$, $A x^* \leq b$ and 
    \begin{align*}
   \vstar \leq  \sum_{i=1}^n f_i(x_i^*) \leq \vstar + (m+1) \max_i \gamma(f_i).
   \end{align*}
\end{restatable}
\begin{proof}
By Lemma~\ref{lem:shapley-folkman-applied-nonconvex-duality}, there exists $\bar{x} \in \R^d$ and set $S \subset \{1, \dots, n\}$ of cardinality at most $m+1$ such that $A\bar{x} \leq b$, $\bar{x}_i \in \domfi$ for $i \not \in S$ and $\bar{x}_i \in \conv \domfi$ such that
    \begin{align*}
        \vstar = \sum_{i \not \in S} f_i(\bar{x}_i) + \sum_{i \in S} f_i^{**}(\bar{x}_i).
    \end{align*}
    In contrast with the proof of Theorem~\ref{thm:nonconvex-duality-gap}, $\bar{x}$ is not necessarily primal feasible since $\domfi$ is not convex. 
    
    For $i \in S$, $\bar{x}_i \in \conv \domfi$ and by Assumption~\ref{ass:nonconvex-domain1} there exists $\hat{x}_i \in \domfi $ such that $A_i \hat{x}_i \leq A_i \bar{x}_i$. Let then $x_i^* = \hat{x}_i \in \domfi$. 
    
    For $i\not \in S$, let $x^*_i = \bar{x}_i$. We then have 
    \begin{align*}
        A x^* = \sum_{i\in S} Ax^*_i + \sum_{i \not \in S} Ax^*_i \leq \sum_{i\in S} A\bar{x}_i + \sum_{i \not \in S} A\bar{x}_i \leq b,
    \end{align*}
    so that $x^*$ is primal feasible, and
    \begin{align*}
        \vstar &\leq \sum_{i=1}^n f_i(x_i^*) \\
        &= \sum_{i\not \in S} f_i(x_i^{*}) + \sum_{i \in S} f_i(x_i^*) \\
        &= \vstar + \sum_{i \in S} (f_i(x_i^*)  - f_i^{**}(\bar{x}_i))\\
        &\leq \vstar + \sum_{i \in S} \sup_{x_i \in \domfi} f_i(x_i) - \inf_{x_i \in \domfi} f_i(x_i) \\
        &\leq \vstar + (m+1) \max_i \gamma(f_i).
    \end{align*}
    The fact that $\inf_{x_i\in \domfi}  f_i(x_i) \leq f_i^{**}(\bar{x}_i)$ follows from the fact that the function $g(x_i) = \inf_{y \in \domfi} f_i(y)$ is a convex minorant of $f_i$ on $\domfi$.
\end{proof}
As for Theorem~\ref{thm:nonconvex-duality-gap}, this theorem is an existence result, and its proof does not yield a constructive algorithm. For MILPs whose bidual~\eqref{eq:bi-dual} has a unique solution, \cite{vujanic2014large} designs an algorithm satisfying the bound of Theorem~\ref{thm:nonconvex-domain-duality-gap}. 

\begin{remark}
\label{rem:comments-on-assumption}
    While Assumption~\ref{ass:nonconvex-domain1} may seem quite restrictive, there exist several real-world applications for which it is satisfied. For example, in the unit commitment problem introduced in \Cref{sec:motivation}, having some $x_i \not \in \domfi$ but $x_i \in \conv \domfi$ means that plant $i$ should be producing between 0 and the minimum amount allowed if turned on. Obviously, this is not possible, but one can simply set the production to the minimum amount allowed, and this will only increase the total amount of electricity produced over all plants, and hence will not violate the demand constraint. The assumption also holds, for example, for supply chain problems~\cite{vujanic2014large} and portfolio optimization~\cite{baumann2013portfolio}.
\end{remark}

If Assumption~\ref{ass:nonconvex-domain1} does not hold, one can instead solve the following perturbed problem
\begin{align}
\tag{\mbox{$\text{CoP}_\perturb$}}
    \begin{split}
    \min_{x_i \in \conv \domfi} & \ \sum_{i=1}^n f_i^{**}(x_i)\\
    \text{subject to }& \ Ax \leq b - \perturb.
    \end{split}
    \label{eq:bi-dual-perturbed}
\end{align}
The next theorem states that for a well chosen value of $\perturb$, there exists a solution of~\eqref{eq:bi-dual-perturbed} whose duality gap can be bounded by the sum of two terms, one depending on the number of constraints $m$, and the second one characterizing the feasibility of~\eqref{eq:bi-dual-perturbed} and the perturbation parameter $\perturb$. This result was proven for MILPs in~\cite{vujanic2016decomposition}, and here we extend it to general nonconvex problems. Recall that $(A_i)_{k, :}$ corresponds to the $k$-th row of the matrix $A_i$.

\begin{restatable}{thm}{NonConvexDomainDualityGapTwo}
    \label{thm:nonconvex-domain-duality-gap2}
    Let $\vstar$ be the dual optimal value as defined in~\eqref{eq:def_v_star} and $\gamma(f_i)$ be as defined in~\eqref{eq:def-max-fi}. Let $\perturb \in \R^m$ be defined as
    \begin{align*}
        \perturb_k = (m+1) \max_i \left( \max_{x_i \in \domfi} ((A_i)_{k, :})^\top x_i - \min_{x_i \in \domfi} ((A_i)_{k, :})^\top x_i \right),
    \end{align*}
    and suppose that there exists $\xi > 0$ and $\hat{x}_i \in \conv \domfi$ for all $i$ such that
    \begin{align*}
        \sum_{i=1}^n A_i \hat{x}_i \leq b - \perturb - n\xi \ones_m.
    \end{align*}
    Then, there exists a solution $x^* \in \R^d$, such that $x^*_i \in \domfi$, $Ax^* \leq b$, and
    \begin{align*}
        v^* \leq \sum_{i=1}^n f_i(x_i^*) \leq v^* + \left( (m+1) + \frac{\norm{\perturb}_\infty}{\xi}\right) \max_i \gamma(f_i).
    \end{align*}
\end{restatable}
\begin{proof}
Let $\vstar_\perturb$ be the optimal value of~\eqref{eq:bi-dual-perturbed}. Applying Lemma~\ref{lem:shapley-folkman-applied-nonconvex-duality} to $\vstar_\perturb$, there exists a set $S \subset \{1, \dots, n\}$ of cardinality at most $m+1$ and $\bar{x} \in \R^d$ such that $A \bar{x} \leq b - \perturb$, $\bar{x}_i \in \domfi$ for $i \not \in S$ and $\bar{x}_i \in \conv (\domfi)$ for $i \in S$ such that
\begin{align*}
    \vstar_\perturb = \sum_{i \in S} f_i^{**}(\bar{x}_i) + \sum_{i \not \in S} f_i(\bar{x}_i)
\end{align*}
Define $x^* \in \R^d$ as $x^*_i = \bar{x}_i$ for $i \not \in S$ and $x^*_i$ be any point in $\domfi$ for $i \in S$. We then have
\begin{align*}
    Ax^* &= \sum_{i\in S} A_i x^*_i + \sum_{i \not \in S} A_i x^*_i = \sum_{i\in S} A_i (x^*_i - \bar{x}_i)  + \sum_{i\in S} A_i \bar{x}_i +  \sum_{i \not \in S} A_i \bar{x}_i \\
    &= \sum_{i\in S} A_i (x^*_i - \bar{x}_i) + \sum_{i=1}^n A_i \bar{x}_i\\
    &\leq \perturb + ( b - \perturb ) = b
\end{align*}
Thus $x^*$ is primal feasible. As in Theorem~\ref{thm:nonconvex-domain-duality-gap}, we can then show that
    \begin{align*}
        \sum_{i=1}^n f_i(x_i^*) \leq \vstar_\perturb +(m+1) \max_i \gamma(f_i).
    \end{align*}
It remains to show that
\begin{align*}
        \vstar_\perturb \leq \vstar + \frac{\norm{\perturb}_\infty}{\xi} \max_i \gamma(f_i).
    \end{align*}
We mirror the proof of~\cite[Theorem 3.3]{vujanic2016decomposition}, except that we make no assumption on the uniqueness of the solution, or on the linearity of the objective functions. Let $\Psi_\perturb$ be defined as $\Psi$ in~\eqref{eq:def-Psi}, except for the perturbed problem, and let $\lambda^*_\perturb = \argmax_{\lambda \geq 0} \Psi_\perturb(\lambda)$. By~\cite[Lemma 1]{nedic2009approximate}, $\hat{x}$ satisfies, for all $\lambda \geq 0$,
    \begin{align*}
        \norm{\lambda^*_\perturb}_1 &\leq \frac{1}{n \xi} \left( \sum_{i=1}^n f_i^{**}(\hat{x}_i) - \Psi_\perturb(\lambda)\right) \\
        &\leq \frac{1}{n \xi} \left( \sum_{i=1}^n f_i^{**}(\hat{x}_i) - \sum_{i=1}^n \inf_{x_i \in \domfi} (f_i(x) + \lambda^\top A_i x_i) - \lambda^\top (b - \perturb)\right).
    \end{align*}
    Plugging $\lambda = 0$, and because $f_i^{**} \leq f_i$, this gives
    \begin{align*}
        \norm{\lambda^*_\perturb}_1 &\leq \frac{1}{\xi} \max_i \gamma(f_i).
    \end{align*}
    Now consider~\eqref{eq:bi-dual-perturbed} as an unperturbed program, and~\eqref{eq:bi-dual} as its perturbed version (perturbed by $+\perturb$). Sensitivity analysis as in~\cite[Section 5.6.2]{boyd2004convex} tells us that since strong duality holds, 
    \begin{align*}
        \vstar_\perturb \leq \vstar + (\lambda^*_\perturb)^\top \perturb \leq \vstar + \norm{\lambda^*_\perturb}_1 \norm{\perturb}_\infty \leq \vstar + \frac{\norm{\perturb}_\infty}{\xi} \max_i \gamma(f_i).
    \end{align*}
\end{proof}    
 As already discussed, the proofs of the duality gap bounds presented in Theorem~\ref{thm:nonconvex-duality-gap}, Theorem~\ref{thm:nonconvex-domain-duality-gap} and Theorem~\ref{thm:nonconvex-domain-duality-gap2} are not constructive. The next section develops an approach to build solutions obtaining those bounds up to an arbitrarily small term.

\section{Solving nonconvex separable problems}
\label{sec:nonconvex-optimization}
We now come back to our original task of finding approximate solutions to problem~\eqref{eq:primal}. Recall that $\vstar$ is the optimal value of the dual~\eqref{eq:dual}.
\subsection{Obtaining primal points}
\label{sec:nonconvex-obtaining-primal-points}
Let us define $z^* \in \R^{1 + m}$ as
\begin{align}
    z^* = \begin{pmatrix}
        \vstar \\ b
    \end{pmatrix}.
\end{align}
Since $\vstar$ is the optimal value of~\eqref{eq:bi-dual}, there exists $x^* \in \R^d$ such that $Ax^* \leq b$ and $\vstar = \sum_{i=1}^n f_i^{**}(x_i^{*})$. In other words,
\begin{align}
    z^* \in \sum_{i=1}^n \left\{ \begin{pmatrix}
        f_i^{**}(x_i) \\
        A_i x_i
    \end{pmatrix} \mid x_i \in \conv \domfi \right\} + \R^{m+1}_+.
\end{align}
Let us define the corresponding sets 
\begin{align*}
    \mathcal{C}_i &= \left\{ \begin{pmatrix}
        f_i^{**}(x_i) \\
        A_i x_i
    \end{pmatrix} \mid x_i \in \conv \domfi \right\}, \quad i=1, \dots,n,\\
    \mathcal{C} &= \sum_{i=1}^n \mathcal{C}_i.
\end{align*}    
We aim to approximate $z^*$ with elements of $\mathcal{C}_i$ where $f_i$ and $f_i^{**}$ match. The next lemma, proved in \Cref{sec:helper-lemmas}, states that this is the case at extreme points of $\mathcal{C}_i$.
    \begin{restatable}{lem}{ExtremePointsMatching}
        Let $\tilde{x}_i \in \conv \domfi$ such that $(f_i^{**}(\tilde{x}_i), A_i \tilde{x}_i)$ is an extreme point of $\mathcal{C}_i$. Then $\tilde{x}_i \in \domfi$ and $f_i(\tilde{x}_i) = f_i^{**}(\tilde{x}_i)$.
    \end{restatable}
    To extract extreme points of $\mathcal{C}_i$, conditional gradient methods are strong candidates. At first sight, a natural objective function is the quadratic loss, giving rise to the following optimization problem
\begin{align}
    \begin{split}
        \min_{z} \ &\frac{1}{2}\norm{z - z^*}^2
        \ \text{ subject to} \  z \in \mathcal{C} + \R^{m+1}_+.
    \end{split}
    \label{eq:FW1-objective}
\end{align}
However, linear minimization on the unbounded set $\mathcal{C} + \R^{m+1}_+$, which is required for conditional gradient methods, is not achievable. Alternatively, one can solve
\begin{align}
    \begin{split}
        \min_{z} \ &\frac{1}{2}\norm{z - z^*}_+^2\  \text{ subject to} \ z \in \mathcal{C}.
    \end{split}
    \label{eq:P1}
\end{align}
\begin{lem}
    Problems~\eqref{eq:FW1-objective} and~\eqref{eq:P1} are equivalent.
\end{lem}
\begin{proof}
    It is quick to check that for any $c \in \mathcal{C}$,
    \begin{align*}
        \norm{c - z^*}_+^2 = \inf_{\delta \in \R^{m+1}_+} \norm{c+ \delta - z^*}^2.
    \end{align*}
    Hence
    \begin{align*}
        \inf_{c \in \mathcal{C}} \left\{ \norm{c - z^*}^2_+ \right\}= \inf_{c \in \mathcal{C}} \inf_{\delta \in \R^{m+1}_+} \left\{ \norm{c+ \delta - z^*}^2\right\} = \inf_{z \in \mathcal{C} + \R^{m+1}_+} \left\{\norm{z - z^*}^2\right\}.
    \end{align*}
\end{proof}

The next lemma shows that any solution of~\eqref{eq:P1} has optimal value $0$ and yields a solution of~\eqref{eq:bi-dual}.
\begin{restatable}{lem}{EquivalenceBetweenFWObjectives}
\label{lem:equivalence-FW1-objectives}
    Suppose $\tilde{z} = \begin{pmatrix}
        \sum_{i=1}^n f_i^{**}(\tilde{x}_i) \\ A \tilde{x}
    \end{pmatrix}$ is a solution of~\eqref{eq:P1}. \\ Then $\norm{\tilde{z} - z^*}_+ = 0$ and $\sum_{i=1}^n f_i^{**}(\tilde{x}_i) = \vstar$ and $A \tilde{x} \leq b$.
\end{restatable}
 \begin{proof}
Let $x^*$ be a solution to the bi-dual problem~\eqref{eq:bi-dual}. We then have 
    \begin{align*}z^* = \begin{pmatrix}
        \vstar \\ b
    \end{pmatrix} = \begin{pmatrix}
        \sum_{i=1}^n f_i^{**}(x_i^*) \\ Ax^*
    \end{pmatrix} + \begin{pmatrix}
        0 \\ b - Ax^*
    \end{pmatrix} \in \sum_{i=1}^n \mathcal{C}_i + \R^{m+1}_+.
    \end{align*}
    In particular, we have
    \begin{align*}
        \frac{1}{2}\norm{\begin{pmatrix}
        \sum_{i=1}^n f_i^{**}(x_i^*) \\ Ax^*
    \end{pmatrix} - z^*}_+^2 = \frac{1}{2}\norm{\begin{pmatrix}
        0 \\ Ax^* - b
    \end{pmatrix}}_+^2 = 0,
    \end{align*}
    since $ A x^* - b\leq 0$. This implies that the minimum value of~\eqref{eq:P1} is 0. In particular, if $\begin{pmatrix}
        \sum_{i=1}^n f_i^{**}(\tilde{x}_i) \\ A \tilde{x}
    \end{pmatrix}$ is a solution of~\eqref{eq:P1}, then $A \tilde{x} \leq b$ and $\sum_{i=1}^n f_i^{**}(\tilde{x}_i) \leq \vstar$. Since $\vstar$ is the optimal value of~\eqref{eq:bi-dual}, we actually have $\sum_{i=1}^n f_i^{**}(\tilde{x}_i) = \vstar$.
\end{proof}

We use the classical Frank Wolfe algorithm to solve~\eqref{eq:P1}. The algorithm starts from some guess $z_0 \in \mathcal{C}$. At each iteration $k$, the gradient of the objective function with respect to $z^k$ is given by $(\alpha_k, g^k) \in \R \times \R^{m}$ with $\alpha_k = \max(z^k_1 - \vstar, 0)$ and $g^k = \max(z^k_{2:m+1} - b, 0)$. The linear minimization step then reads
\begin{align}
    \label{eq:FW1-LMO}
    s^k \in \argmin_{z \in \sum_{i=1}^n \mathcal{C}_i} z^\top (\alpha_k, g^k).
\end{align}
This amounts to solving, for each $i \in \{1, \dots, n\}$,
\begin{align}
    \label{eq:FW1-LMO-2}
    y_i^k \in \argmin_{x_i \in \conv \domfi} \alpha_k f_i^{**}(x_i) + (g^k)^\top A_i x_i,
\end{align}
and setting $s^k = \sum_{i=1}^n y_i^k$. If $\alpha_k = 0$, the above is a linear minimization over $\conv \domfi$. Otherwise, $\alpha_k > 0$, and we can can rewrite~\eqref{eq:FW1-LMO-2} as
\begin{align*}
    y_i^k \in \argmax_{x_i \in \conv \domfi} \left( - \frac{A_i^\top g^k}{\alpha_k}\right)^\top x_i - f_i^{**}(x_i) = \partial f_i^*\left( - \frac{A_i^\top g^k}{\alpha_k}\right).
\end{align*}
The last few lines are central to our approach. They tell us that the key subproblem in conditional gradient methods, namely the linear minimization oracle~\eqref{eq:FW1-LMO}, is as tractable as computing the conjugate of the functions $f_i$.

\begin{remark}
    It is interesting to look into the computation of~\eqref{eq:FW1-LMO-2} when the function $f_i$ is concave and its domain $\domfi$ is a polytope. In that case, the convex envelope of $f_i$ on $\domfi$ is \textit{vertex polyhedral}, meaning that it is equal to the convex envelope of $f_i$ restricted to the vertices of $\domfi$~\cite{falk1976successive}. Importantly here, the minimum of the convex envelope of such functions, and therefore also the computation of~\eqref{eq:FW1-LMO-2}, then comes down to solving a linear system and can thus be obtained efficiently~\cite{tardella2004existence}. Note that functions with a vertex polyhedral convex envelope are not restricted to concave functions but also include, for example, the more general edge-concave functions~\cite{tardella2004existence} and multilinear functions on certain polytopes~\cite{rikun1997convex}.
\end{remark}

We summarize our full approach in~\Cref{alg:FW1}. Note that while there always exists an extreme point of $\mathcal{C}_i$ which solves~\eqref{eq:FW1-LMO-2}, a non-extreme point might yield the same objective value. Without loss of generality, we assume that we always get an extreme point, and hence $y_i^k \in \domfi$ and $f_i(y_i^k) = f_i^{**}(y_i^k)$. In the applications we consider in section \Cref{sec:numerical-results}, this is never an issue.

 At any iteration $k$ of the algorithm, the suboptimality gap $\norm{z^k - z^*}_+$ can be used as stopping criterion. Alternatively, the quality of the approximation can be bounded using classical convergence results on the Frank-Wolfe method.

\begin{thm}
\label{thm:FW1-rate}
    Suppose we run~\Cref{alg:FW1} for $K$ iterations with $\eta_k = \frac{2}{k+2}$ or $\eta_k = \min\left(1, \frac{(v^k)^\top(z^k - s^k)}{\norm{z^k - s^k}^2}\right)$, where $v^k = \max(z^k - z^*, 0)$. Then the output satisfies
    \begin{align*}
        \norm{z^K - z^*}^2_+ \leq \frac{4\diamC^2}{K+1},
    \end{align*}
    where $\diamC = \max_{z, z' \in \mathcal{C}} \norm{z - z'}$.
\end{thm}
\begin{proof}
    This is a direct consequence of the analysis of Frank-Wolfe algorithm, see for example~\cite[Theorem 2.3]{jaggi2011sparse} for the stepsize $\eta_k = \frac{2}{k+2}$ or~\cite{pedregosa2020linearly} for the other choice of $\eta_k$. Note that the feasible set $\mathcal{C}$ we are considering is not necessarily convex, and hence it might be that $z^K \not \in \mathcal{C}$. This is however not an issue for the presented bound to hold.
\end{proof}

After running~\Cref{alg:FW1} for $K$ iterations, the output $z^K$ is a convex combination of the vectors $\{s_k\}_{k=0}^{K-1}$. In other words, for $\gamma \in \Delta_K$, we have
 \begin{align}
    \label{eq:FW1-output}
     z^K = \sum_{k=0}^{K-1} \gamma_k s_k = \sum_{k=0}^{K-1} \gamma_k \sum_{i=1}^n \begin{pmatrix}
         f_i(y_i^k) \\ A_i y_i^k
     \end{pmatrix}.
 \end{align}

\begin{algorithm}
\caption{Frank Wolfe algorithm to obtain primal points}
\label{alg:FW1}
\begin{algorithmic}[1]
\STATE{ Input: $z^* \in \R^{1 + m}$, number of iterations $K$.}
\STATE{For all $i=1, \dots, n$, choose $y_i^0 \in \domfi$ such that $f_i(y_i^0) = f_i^{**}(y_i^0)$.}
\STATE{Set $z^0 = \sum_{i=1}^n \begin{pmatrix}
    f_i(y_i^0)\\ A_i y_i^0
\end{pmatrix}$}
\FOR{$k=0, \dots, K-1$}
\STATE{Set $\alpha_k = \max(z^k_1 - \vstar, 0)$ and $g^k = \max(z^k_{2:m+1} - b, 0)$.}
\FOR{$i=1, \dots, n$}
\STATE{\textbf{if} $\alpha_k = 0$ then \begin{align*}
    y_i^k \in \argmin_{x_i \in \conv \domfi} (A_i^\top g^k)^\top x_i
\end{align*}}
\STATE{\textbf{else}\begin{align*}
    y_i^k \in \partial f_i^*\left( - \frac{A_i^\top g^k}{\alpha_k}\right)
\end{align*}}
\ENDFOR
\STATE{Set $s^k = \sum_{i=1}^n \begin{pmatrix}
    f_i(y_i^k) \\ A_i y_i^k
\end{pmatrix}$.}
\STATE{Pick some stepsize $\eta_k > 0$.}
\STATE{$z^{k+1} = ( 1 - \eta_k) z_k + \eta_k s_k$.}
\ENDFOR
\RETURN $z^K$
\end{algorithmic}
\end{algorithm}

\begin{remark}
   \Cref{alg:FW1} optimizes a (non-strongly) convex, smooth objective function. In this scenario, Frank-Wolfe algorithms cannot achieve better than $O\left(\frac{1}{K}\right)$ convergence rate~\cite{lan2013complexity}. Therefore, another conditional gradient method would not achieve better performance than \Cref{alg:FW1}, at least in theory. Of course, one is free to implement any variant of Frank-Wolfe, for example pairwise~\cite{mitchell1974finding,lacoste2015global}, away-steps~\cite{wolfe1970convergence}, or classical Frank-Wolfe with a backtracking linesearch~\cite{pedregosa2020linearly}. The only requirement is that it must be possible to keep track of the coefficients in the convex representation of the final result $z^K$. 
\end{remark}

\subsection{Trimming the solution}
\label{sec:nonconvex-trimming}
Let us define $w^K \in \R^{1 + m + n}$ and $\AK\subset \R^{1 +m+n}$ as
\begin{align}
\label{eq:wK-and-AK}
    w^K = \begin{pmatrix}
        z^K \\ \ones_n
    \end{pmatrix}, \ \text{ and } \ \AK = \left\{ \begin{pmatrix}
        f_i(y_i^k) \\ A_i y_i^k \\ e_i
    \end{pmatrix} \mid k = 0, \dots, K-1, i = 1, \dots n\right\}.
\end{align}
From~\Cref{eq:FW1-output}, we get that
\begin{align}
    \label{eq:FW1-output-big}
    w^K = \sum_{k=0}^{K-1} \gamma_k \sum_{i=1}^n \begin{pmatrix}
         f_i(y_i^k) \\ A_i y_i^k \\ e_i
     \end{pmatrix},
\end{align}
where $\{e_i\}_{i=1}^n$ is the canonical basis of $\R^n$. In other words, $w^K \in \po(\AK)$.

 We can thus apply \Cref{alg:Constructive-caratheodory} with $\wstar = w^K$. From Theorem~\ref{thm:Constructive-caratheodory}, this yields a vector $\alpha \in \R^{1 + m+n}_+$ such that
\begin{align}
\label{eq:trimming-result-regrouping}
    w^K = \sum_{l=1}^{n+m+1}\alpha_l \begin{pmatrix} 
            f_{i_l}(y_{i_l}^{k_l}) \\ A_{i_l} y_{i_l}^{k_l} \\ e_{i_l}
        \end{pmatrix} = \sum_{i=1}^n \sum_{l \in I_i} \alpha_l \begin{pmatrix}
        f_i(y_i^{k_l}) \\ A_i y_i^{k_l} \\ e_i
    \end{pmatrix},
\end{align}
where we regrouped the terms in $i$ by defining $I_i := \{ l \mid i_l = i\}$ for each $i=1, \dots, n$.  This yields the following theorem.
\begin{prop}
    \label{prop:exact-caratheodory-specified}
    Suppose we run \Cref{alg:Constructive-caratheodory} with input $\AK$ and $w^K$ defined in~\eqref{eq:FW1-output-big}. It outputs a vector $\alpha \in \R^{n+m+1}_+$ such that
    \begin{align}
    \label{eq:trim-output-1}
        w^K = \sum_{i=1}^n \sum_{l \in I_i} \alpha_l \begin{pmatrix}
        f_i(y_i^{k_l}) \\ A_i y_i^{k_l} \\ e_i
    \end{pmatrix}.
    \end{align}
    In particular, for all $i$, we have $\sum_{l \in I_i} \alpha_l = 1$. Moreover, $|I_i| > 1$ for at most $m+1$ indices $i$.  
\end{prop}
\begin{proof}
    The fact that $\sum_{l \in I_i} \alpha_l = 1$ follows from the last $n$ coordinates of $w^K$ being equal to $1$. In particular, we must have $|I_i| \geq 1$ for all $i$. Since $\sum_{i=1}^n |I_i| = n+m+1$, there can only be at most $m+1$ indices $i$ such that $|I_i| > 1$.
\end{proof}

\begin{remark}
    If \Cref{alg:Constructive-caratheodory} is applied with a dense QR decomposition and column addition/deletion as described before Lemma~\ref{lem:exact-carath-flop-count}, the complexity is $O(nK (n+m)^2)$. Alternatively, since the vectors in $\AK$ are sparse (only one coordinate among the last $n$ is nonzero), sparse solvers might be more efficient in practice.
\end{remark}

\subsection{Final reconstruction}
\label{sec:final-reconstruction}
 In this section, we explore different ways to reconstruct a solution of the primal problem~\eqref{eq:primal} from the decomposition~\eqref{eq:trim-output-1}. We start with the following lemma, which is straightforward from the convergence rate derived in Theorem~\ref{thm:FW1-rate}. The reason we split the coordinates is that it will be useful to differentiate between the coordinates related to function values and the coordinates related to the affine constraints later on.
\begin{lem}
    \label{lem:reconstruction-technical-lemma}
    Let $\vstar$ be the optimal value of~\eqref{eq:dual} and~\eqref{eq:bi-dual}. Assume that we run~\Cref{alg:FW1} for $K$ iterations, followed by \Cref{alg:Constructive-caratheodory} to trim the number of elements. Then
    \begin{align*}
        &w^K_1 \leq v^* + \frac{2\diamC}{\sqrt{K+1}},\\
        &\norm{w^K_{2:m+1} - b}_+ \leq \frac{2\diamC}{\sqrt{K+1}}.
    \end{align*}
\end{lem}
\begin{proof}
    Since $z^K$, namely the output of \Cref{alg:FW1}, consists of the first $1+m$ coordinates of $w^K$, we have in particular $w^K_1 = z^K_1$ and $w^K_{2:m+1} = z^K_{2:m+1}$. The result then follows from the convergence rate from Theorem~\ref{thm:FW1-rate}.
\end{proof}

 \subsubsection{Convex domains}
 When the function domains are convex, a natural candidate solution is $\bar{x} \in \R^d$ defined as
\begin{align}
    \label{eq:reconstruction-CC}
    \bar{x}_i = \sum_{l \in I_i} \alpha_l y_{i}^{k_l}, \quad i=1, \dots, n.
\end{align}

The next result, which we consider to be the main result of this work, recovers the bound of Theorem~\ref{thm:nonconvex-duality-gap} for the proposed solution~\eqref{eq:reconstruction-CC}, up to an additional term which decreases with $K$.

\begin{thm}[Main result]
\label{thm:convergence-convex-domain}
    Suppose that $\domfi$ is convex for all $i=1, \dots, n$, and let $\vstar$ be the optimal value of~\eqref{eq:dual} and~\eqref{eq:bi-dual}. Assume that we run~\Cref{alg:FW1} for $K$ iterations, followed by \Cref{alg:Constructive-caratheodory} to trim the number of elements. Let $\bar{x}\in \R^d$ be the final solution defined in~\eqref{eq:reconstruction-CC}. Then $\bar{x}$ satisfies
\begin{align*}
        \sum_{i=1}^n f_i(\bar{x}_i) &\leq \vstar +  \frac{2\diamC}{\sqrt{K+1}}+  (m+1) \max_{i} \rho (f_i),\\
         \norm{\sum_{i=1}^n A_i \bar{x}_i - b}_+ &\leq \frac{2\diamC}{\sqrt{K+1}}.
    \end{align*}
\end{thm}
\begin{proof}
Observe that
\begin{align}
\label{eq:final-CC-computable-bound}
\begin{split}
    \sum_{i=1}^n f_i(\bar{x}_i) - \vstar &=  \left(\sum_{i=1}^n f_i(\bar{x}_i) - w^K_1 \right)+ \left(w^K_1 - \vstar\right)\\
    \norm{\sum_{i=1}^n A_i\bar{x}_i - b}_+ &\leq  \norm{\sum_{i=1}^n A_i \bar{x}_i - w^K_{2:m+1}}_+ + \norm{w^K_{2:m+1} - b}_+
    \end{split}
\end{align}
We can apply Lemma~\ref{lem:reconstruction-technical-lemma} to bound the last term of each equation. Moreover,
    \begin{align*}
    \sum_{i=1}^n f_i(\bar{x}_i) - w^K_1  &= \sum_{i=1}^n f_i(\bar{x}_i) - \sum_{i=1}^n \sum_{l\in I_i} \alpha_l f_i(y_{i}^{k_l}) \\
        &= \sum_{i, |I_i| > 1}  \left( f_i\left(\sum_{l \in I_i} \alpha_l y_{i}^{k_l}\right) - \sum_{l \in I_i}\alpha_l f_i(y_{i}^{k_l}) \right)\\
        &\leq \sum_{i, |I_i| > 1} \rho(f_i) \\
        &\leq (m+1) \max_{i} \rho(f_i),
    \end{align*}
    and
    \begin{align*}
         \sum_{i=1}^n A_i \bar{x}_i
         - w^K_{2:m+1} &=  \sum_{i=1}^n A_i \sum_{l\in I_i} \alpha_l y_{i}^{k_l}
         - \sum_{i=1}^n \sum_{l \in I_i} \alpha_l A_i y_{i}^{k_l} = 0.
    \end{align*}
\end{proof}

\subsubsection{General domains}
\label{sec:reconstruction-general-domains}
When the domains $\domfi$ are not necessarily convex, the solution $\bar{x}$ defined in~\eqref{eq:reconstruction-CC} is in general not feasible, as for $i$ such that $|I_i| > 1$,  $\bar{x}_i$ is a non-trivial convex combination of elements of $\domfi$. In this section, we suggest several approaches to construct feasible solutions.

\paragraph{Reconstruction via sampling} Our first reconstruction scheme is a randomized approach. We suggest setting the solution $\bar{x}_i$ as $y_i^{k_l}$ with probability $\alpha_l$. In other words, for each $i=1, \dots, n$, define $L_i$ as the random variable taking values in $I_i$ such that for all $l \in I_i$, $P(L_i = l) = \alpha_l$. The reconstruction scheme is then
\begin{align}
\label{eq:reconstruction-sampled}
\begin{split}
    &l_i \sim L_i \\
    &\bar{x}_i = y_i^{k_{l_i}}
\end{split} \quad i=1, \dots, n.
\end{align}
Since $y_i^{k_l} \in \domfi$ for all $i$ and $l \in I_i$, we immediately get that $\bar{x}_i \in \domfi$ for all $i$. We also get the following duality gap bound on the obtained solutions. In particular, the obtained solution satisfies the constraint in expectation (up to an arbitrarily small term).

\begin{thm}
\label{thm:convergence-sampled-solution}
   Let $\vstar$ be the optimal value of~\eqref{eq:dual} and~\eqref{eq:bi-dual}. Assume that we run~\Cref{alg:FW1} for $K$ iterations, followed by \Cref{alg:Constructive-caratheodory} to trim the number of elements. Let $\bar{x}\in \R^d$ be the final solution defined in~\eqref{eq:reconstruction-sampled}. Then $\bar{x}$ satisfies
\begin{align*}
        \sum_{i=1}^n f_i(\bar{x}_i) &\leq \vstar +  \frac{2\diamC}{\sqrt{K+1}}+  (m+1) \max_{i} \gamma (f_i),\\
         \E\left[\sum_{i=1}^n A_i \bar{x}_i - b\right] &\leq \frac{2\diamC}{\sqrt{K+1}}.
    \end{align*}
\end{thm}
\begin{proof}
Observe that
\begin{align}
\begin{split}
    \sum_{i=1}^n f_i(\bar{x}_i) - \vstar &=  \left(\sum_{i=1}^n f_i(\bar{x}_i) - w^K_1 \right)+ \left(w^K_1 - \vstar\right)\\
    \sum_{i=1}^n A_i\bar{x}_i - b &= \left(\sum_{i=1}^n A_i \bar{x}_i - w^K_{2:m+1} \right) +  \left(w^K_{2:m+1} - b\right)\\
    &\leq \left(\sum_{i=1}^n A_i \bar{x}_i - w^K_{2:m+1} \right) + \norm{w^K_{2:m+1} - b}_+.
    \end{split}
\end{align}
Lemma~\ref{lem:reconstruction-technical-lemma} gives a bound on the last term of each equation. Moreover, we have
        \begin{align*}
        \sum_{i=1}^n f_i(\bar{x}_i) - w^K_1 &= \sum_{i=1}^n f_i(\bar{x}_i) - \sum_{i=1}^n \sum_{l\in I_i} \alpha_l f_i(y_{i}^{k_l})  \\
        &= \sum_{i, |I_i| > 1} \left(f_i(\bar{x}_i) - \sum_{l\in I_i} \alpha_l f_i(y_{i}^{k_l})\right)\\
        &\leq \sum_{i, |I_i| > 1} \gamma(f_i) \\
        &\leq (m+1) \max_{i} \gamma(f_i),
    \end{align*}
    and
    \begin{align*}
         \E\left[\sum_{i=1}^n A_i \bar{x}_i
         - w^K_{2:m+1}\right] &=  \sum_{i=1}^n A_i \E\left[\bar{x}_i\right]
         - w^K_{2:m+1} \\
         &= \sum_{i=1}^n A_i \sum_{l\in I_i} \alpha_l y_{i}^{k_l}
         - \sum_{i=1}^n \sum_{l \in I_i} \alpha_l A_i y_{i}^{k_l} = 0.
    \end{align*}
\end{proof}

\paragraph{Reconstruction under additional domain assumption}
The above reconstruction gives a solution which is feasible in expectation. In particular, one can build solutions of the form~\eqref{eq:reconstruction-sampled} by repeatedly sampling until a good enough solution is obtained. However, there is no guarantee that a fully feasible solution will be obtained with that procedure. In order to construct such a solution, we require additional assumptions on the feasible set.

We already explored one common such assumption in \Cref{sec:duality-gap-estimation}, namely Assumption~\ref{ass:nonconvex-domain1}. We restate it here for completeness.

\AssNonConvexDomain*
Recall that we discussed in \Cref{rem:comments-on-assumption} several real-world applications for which this assumption holds. We can then construct a final solution as follows. If index $i$ is such that $|I_i| > 1$, then by Assumption~\ref{ass:nonconvex-domain1} there exists $\hat{x}_i \in \domfi$ such that
\begin{align}
    \label{eq:final-reconstruction-nonconvex-nontrivial-indices-inequality}
    A_i \hat{x}_i \leq A_i \sum_{l \in I_i} \alpha_l y_i^{k_l}
\end{align}
We can then build a final solution as
\begin{align}
\label{eq:reconstruction-nonconvex1}
\begin{split}
    \bar{x}_i = \begin{cases} \sum_{l \in I_i} \alpha_l y_{i}^{k_l} & \text{ if } |I_i| = 1,\\
    \hat{x}_i &\text{ if } |I_i| > 1.
    \end{cases}
\end{split}
\end{align}

The next theorem states that the resulting solution recovers the duality gap bound of Theorem~\ref{thm:nonconvex-domain-duality-gap}, up to an additional term decreasing with $K$.

\begin{restatable}{thm}{NonConvexDomainConvergenceOne}
\label{thm:convergence-noncconvex-domain1}
    Let $\vstar$ be the optimal value of~\eqref{eq:dual} and~\eqref{eq:bi-dual} and suppose that Assumption~\ref{ass:nonconvex-domain1} holds. Assume that we run~\Cref{alg:FW1} for $K$ iterations, followed by \Cref{alg:Constructive-caratheodory} to trim the number of elements. Let $\bar{x}\in \R^d$ be the final solution defined in~\eqref{eq:reconstruction-nonconvex1}. Then $\bar{x}$ satisfies
\begin{align*}
        \sum_{i=1}^n f_i(\bar{x}_i) &\leq \vstar +  \frac{2\diamC}{\sqrt{K+1}}+  (m+1) \max_{i} \gamma (f_i),\\
         \norm{\sum_{i=1}^n A_i \bar{x}_i - b}_+ &\leq \frac{2\diamC}{\sqrt{K+1}}.
    \end{align*}
\end{restatable} \begin{proof}
Observe that
\begin{align}
\begin{split}
    \sum_{i=1}^n f_i(\bar{x}_i) - \vstar &=  \left(\sum_{i=1}^n f_i(\bar{x}_i) - w^K_1 \right)+ \left(w^K_1 - \vstar\right),\\
    \norm{\sum_{i=1}^n A_i\bar{x}_i - b}_+ &\leq  \norm{\sum_{i=1}^n A_i \bar{x}_i - w^K_{2:m+1}}_+ + \norm{w^K_{2:m+1} - b}_+.
    \end{split}
\end{align}
We can apply Lemma~\ref{lem:reconstruction-technical-lemma} to bound the last term of each equation. Moreover, as in the proof of Theorem~\ref{thm:convergence-sampled-solution}, we have
    \begin{align*}
    \sum_{i=1}^n f_i(\bar{x}_i) - w^K_1  \leq  (m+1) \max_{i} \gamma(f_i).
    \end{align*}
    Now,
    \begin{align*}
        \norm{ \sum_{i=1}^n A_i \bar{x}_i
         - w^K_{2:m+1}}_+ &=  
         \norm{ \sum_{i=1}^n A_i \bar{x}_i
         - \sum_{i=1}^n \sum_{l \in I_i} \alpha_l A_i y_{i}^{k_l}}_+ \\
         &= \norm{\sum_{i, |I_i| > 1} \left(A_i \hat{x}_i - \sum_{l \in I_i} \alpha_l A_i y_{i}^{k_l}\right)}_+ \\
         &= 0,
    \end{align*} 
    where the last equality is a consequence of~\eqref{eq:final-reconstruction-nonconvex-nontrivial-indices-inequality}.
\end{proof}

\paragraph{Reconstruction after solving perturbed problem}
\label{sec:final-reconstruction-perturbed}
Assumption~\ref{ass:nonconvex-domain1} can sometimes be restrictive. In such cases, following~\cite{vujanic2016decomposition}, we consider a perturbed version of~\eqref{eq:bi-dual}, namely~\eqref{eq:bi-dual-perturbed}, with $\perturb \in \R^m$ defined as 
\begin{align}
    \label{eq:rho-definition}
    \perturb_k := (m+1) \max_i \left( \max_{x_i \in \domfi} ((A_i)_{k, :})^\top x_i - \min_{x_i \in \domfi} ((A_i)_{k, :})^\top x_i \right).
\end{align}
This will ensure that the error made when constructing a final solution is small enough so that the solution is still primal feasible. We also need the following technical assumption, which is a Slater-type condition on the perturbed problem~\eqref{eq:bi-dual-perturbed}.
\begin{assumption}
\label{ass:nonconvex-domain2}
    For all $i=1, \dots, n$, there exists $\xi > 0$ and $\hat{x}_i \in \conv \domfi$ such that
    \begin{align*}
        \sum_{i=1}^n A_i \hat{x}_i \leq b - \perturb - n\xi \ones_m.
    \end{align*}
\end{assumption}
Finally, we suggest the following reconstruction scheme
    \begin{align}
    \label{eq:reconstruction-max}
    \begin{split}
    l_i &= \argmax_{l \in I_i} \alpha_l,\\
    \bar{x}_i &=   y_{i}^{k_{l_i}},
    \end{split}\quad i=1, \dots, n.
\end{align}
Taking the maximum over all coefficients in the convex combinations allows us to prove a stronger bound than in Theorem~\ref{thm:nonconvex-domain-duality-gap2}.

\begin{restatable}{thm}{NonConvexDomainConvergenceTwo}
\label{thm:convergence-noncconvex-domain2}
Let $\vstar$ be the optimal value of~\eqref{eq:dual} and~\eqref{eq:bi-dual} and suppose that Assumption~\ref{ass:nonconvex-domain2} holds. Suppose that we apply the method of \Cref{sec:nonconvex-obtaining-primal-points,sec:nonconvex-trimming} to problem~\eqref{eq:bi-dual-perturbed} with $\perturb$ defined in~\eqref{eq:rho-definition}. Assume that we run~\Cref{alg:FW1} for $K$ iterations, followed by \Cref{alg:Constructive-caratheodory} to trim the number of elements. Let $\bar{x}\in \R^d$ be the final solution defined in~\eqref{eq:reconstruction-max}. Finally, let $q = | \{ i\in \{1, \dots, n\} \mid |I_i| > 1\}|$ be the number of indices for which the output of the Carath\'edory algorithm gives a non-trivial convex combination. Then $\bar{x}$ satisfies
\begin{align*}
        \sum_{i=1}^n f_i(\bar{x}_i) &\leq \vstar +  \frac{2\diamC}{\sqrt{K+1}}+  \left( \frac{q(m+1)}{q + (m+1)} + \frac{\norm{\perturb}_\infty}{\xi} \right) \max_{i} \gamma (f_i),\\
         \norm{\sum_{i=1}^n A_i \bar{x}_i - b}_+ &\leq \frac{2\diamC}{\sqrt{K+1}}.
    \end{align*}
\end{restatable}
\begin{proof}
Let $\vstar_\perturb$ be the optimal value of~\eqref{eq:bi-dual-perturbed}.
Observe that
\begin{align}
\begin{split}
    \sum_{i=1}^n f_i(\bar{x}_i) - \vstar &=  \left(\sum_{i=1}^n f_i(\bar{x}_i) - w^K_1 \right)+ \left(w^K_1 - \vstar_\perturb\right) + \left( \vstar_\perturb - \vstar \right)\\
    \sum_{i=1}^n A_i\bar{x}_i - b &\leq  
    \left(\sum_{i=1}^n A_i \bar{x}_i - w^K_{2:m+1}\right) - \perturb  + \left(w^K_{2:m+1} - (b - \perturb)\right) \\
    &\leq \left(\sum_{i=1}^n A_i \bar{x}_i - w^K_{2:m+1}\right) - \perturb  + \ones_{m}\norm{w^K_{2:m+1} - (b - \perturb)}_+
    \end{split}
\end{align}
Since we assumed that we applied the method to the perturbed problem~\eqref{eq:bi-dual-perturbed}, we can apply Lemma~\ref{lem:reconstruction-technical-lemma} to get 
\begin{align*}
    \left(w^K_1 - \vstar_\perturb\right) &\leq \frac{2\diamC}{\sqrt{K + 1}} ,\\
    \norm{w^K_{2:m+1} - (b - \perturb)}_+ &\leq \frac{2\diamC}{\sqrt{K + 1}}.
\end{align*}
Now,
\begin{align*}
    \sum_{i=1}^n A_i \bar{x}_i - w^K_{2:m+1} &= \sum_{i=1}^n A_i \bar{x}_i - \sum_{i=1}^n \sum_{l \in I_i} \alpha_l A_i y_{i}^{k_l} \\
    &= \sum_{i, |I_i| > 1} \left( A_i \bar{x}_i - \sum_{l \in I_i} \alpha_l A_i y_{i}^{k_l}\right) \\
    &\leq \perturb.
\end{align*}
This proves the inequalities on the feasibility. Moreover, we showed in Theorem~\ref{thm:nonconvex-domain-duality-gap2} that 
\begin{align*}
    \vstar_\perturb \leq \vstar + \frac{\norm{\perturb}_\infty}{\xi} \max_i \gamma(f_i).
\end{align*}
Therefore, it only remains to show that
\begin{align*}
    \sum_{i=1}^n f_i(\bar{x}_i) - w^K_1  \leq \frac{q(m+1)}{q + (m+1)} \max_i \gamma(f_i).
\end{align*}
Indeed, we have
\begin{align*}
        \sum_{i=1}^n f_i(\bar{x}_i) - w^K_1  &= \sum_{i=1}^n f_i(\bar{x}_i) - \sum_{i=1}^n \sum_{l\in I_i} \alpha_l f_i(y_{i}^{k_l})  \\
        &= \sum_{i, |I_i| > 1}  \left(f_i(y_i^{k_{l_i}}) - \sum_{l \in I_i}\alpha_l f_i(y_{i}^{k_l})\right) \\
        &= \sum_{i, |I_i| > 1} \left( (1 - \alpha_{l_i}) f_i(y_{i}^{k_{l_i}}) - \sum_{l \in I_i, l\not = l_i}\alpha_l f_i(y_i^{k_{l}})\right) \\
        &=  \sum_{i, |I_i| > 1} (1 - \alpha_{l_i})  \left( f_i(y_{i}^{k_{l_i}}) - \sum_{l \in I_i, l\not = l_i} \frac{\alpha_l}{1 - \alpha_{l_i} } f_i(y_i^{k_{l}}) \right)  \\
        &\leq \max_{i} \gamma(f_i) \sum_{i, |I_i| > 1} (1 - \alpha_{l_i}).
    \end{align*}
    We now show that
    \begin{align*}
        \sum_{i, |I_i| > 1} (1 - \alpha_{l_i}) \leq \frac{q(m+1)}{q + (m+1)}.
    \end{align*}
    First observe that since $\alpha_{l_i} = \max_{l \in I_i} \alpha_l$ and $\sum_{l\in I_i} \alpha_l = 1$, we must have $\alpha_{l_i} \geq \frac{1}{|I_i|}$. Moreover we have $\sum_{i, |I_i| > 1} (|I_i| - 1) \leq (m+1)$. We therefore look for an upper bound on the program:
        \begin{align*}
            \max_{u_i \geq 2} \quad & \sum_{i=1}^q ( 1 - \frac{1}{u_i})\\
            \text{s.t. } & \sum_{i=1}^q (u_i - 1) \leq m+1.
        \end{align*}
        Since the program is concave and invariant under permutation, the solution must have all $u_i$'s equal so that it can be rewritten
        \begin{align*}
            \max_{u \geq 2} \quad & q ( 1 - \frac{1}{u})\\
            \text{s.t. } & 
             q(u - 1) \leq m+1.
        \end{align*}
        The Lagrangian reads
        \begin{align*}
            L(u, \lambda, \mu) = q( 1 - \frac{1}{u}) +  \lambda (u - 2) + \mu \left( m+1 - q(u-1) \right).
        \end{align*}
        Maximization with respect to $u$ gives
        \begin{align*}
            \frac{q}{u^2} + \lambda - \mu q = 0 \Rightarrow \frac{1}{u} = \sqrt{\mu - \lambda/q},
        \end{align*}
        and the dual then reads
        \begin{align*}
            \min_{\mu \geq \lambda / q \geq 0} \quad & q - q\sqrt{\mu - \lambda/q} - 2 \lambda + \mu(m+1) + \mu q - q\sqrt{\mu - \lambda/q}.
        \end{align*}
        By the change of variable $\gamma = \sqrt{\mu - \lambda/q}$, we get
        \begin{align*}
            \min_{\gamma, \lambda \geq 0} \quad & q - 2q\gamma - 2 \lambda + (\gamma^2 + \frac{\lambda}{q})(m+1) + (\gamma^2 + \frac{\lambda}{q}) q 
        \end{align*}
        First order optimality gives $\lambda = 0$ (since $q\leq m+1$) and $\gamma = \frac{q}{q + (m+1)}$. The corresponding optimal value is then $\frac{q(m+1)}{q + (m+1)}$.
\end{proof}
In this section, we saw that $q \leq m+1$. Plugging in $q= (m+1)$ in the above theorem gives a term $\frac{m+1}{2}$ in the bound, which improves the term in Theorem~\ref{thm:nonconvex-domain-duality-gap2} by a factor 2.

\begin{remark}
    The approach presented in this section constructs solutions which meet the duality gap bounds presented in~\Cref{sec:duality-gap-estimation}, up to an additional term $O(1/\sqrt{K})$, both in terms of functions values and feasibility. We argue that this not a major limitation. Indeed, in order to obtain a fully feasible solution, in practice one can solve a slightly perturbed version of the original problem. We describe one such way in the experimental section. Moreover, the additional term $O(1/\sqrt{K})$ in the function value bounds is typically negligible compared to the term in $(m+1)\max_i \rho(f_i)$ or $(m+1) \max_i \gamma(f_i)$.
\end{remark}

\section{Extension to approximate Carath\'eodory representations}
\label{sec:extension-approx-carath}

In \Cref{sec:nonconvex-trimming}, we showed how one can trim the result of the output of the first Frank-Wolfe algorithm by applying \Cref{alg:Constructive-caratheodory}, which computes an exact conic Cara\-th\'eodory representation of $w^K$. In this section, we show that one can also compute cheaper approximate Carath\'eodory representations of $w^K$. Those approximate representations can be obtained by running a conditional gradient method. Those methods have the advantage of reducing the dependency on $m$ in the duality gap bound, at the expense of an exponentially decreasing error term which depends on the number of iterations of the conditional gradient method.

First let us recall that, after running \Cref{alg:FW1}, we obtain $z^K \in \R^{1+m}$ with
\begin{align*}
    z^K = \sum_{k=0}^{K-1} \gamma_k \sum_{i=1}^n \begin{pmatrix}
         f_i(y_i^k) \\ A_i y_i^k
     \end{pmatrix},
\end{align*}
where $\sum_{k=0}^{K-1} \gamma_k = 1$. Moreover, recall from~\eqref{eq:wK-and-AK} that $w^K$ and $\AK$ are defined as
\begin{align}
    &w^K = \begin{pmatrix}
        z^K \\ \ones_n
    \end{pmatrix} = \sum_{k=0}^{K-1} \gamma_k \sum_{i=1}^n \begin{pmatrix}
         f_i(y_i^k) \\ A_i y_i^k \\ e_i
     \end{pmatrix} ,\\ &\AK = \left\{ \begin{pmatrix}
        f_i(y_i^k) \\ A_i y_i^k \\ e_i
    \end{pmatrix} \mid k = 0, \dots, K-1, i = 1, \dots n\right\},
\end{align}
In particular, we have $w^K/n \in \conv \AK$. Instead of computing an exact Cara\-th\'eodory decomposition of $w^K/n$, we look for an approximate decomposition, namely a set of vectors whose convex combination is a good approximation of $w^K$. Following on ideas from~\cite{combettes2023revisiting}, we therefore suggest applying a Frank-Wolfe algorithm to the following problem
\begin{align}
    \argmin_{w \in \conv \AK} h(w) := \frac{1}{2} \norm{w - \frac{w^K}{n}}^2.
\end{align}
The reason behind using the conditional gradient method is that the output is a convex combination of vectors in $\AK$, and therefore an approximate Carath\'eodory representation of $w^K/n$. The approximation error will be related to the convergence rate of the conditional gradient method. In this section, we will focus on the fully-corrective Frank-Wolfe (FCFW) algorithm, as it gives the best theoretical results~\cite{combettes2023revisiting}, but we will see how one can extend to other conditional gradient methods at the end of the section. The FCFW method starts with some $w^{ K, 0} \in \AK$, sets $s^0 = w^{K, 0}$ and defines an active set $\mathcal{S}^0 = \{s^{0}\}$. At each iteration $t\geq 1$, given some $w^{K, t-1}$ and an active set $\mathcal{S}^{t-1}$, the method makes a call to the linear minimization oracle
\begin{align}
    \label{eq:fcfw-lmo}
    s^t \in \argmin_{s \in \conv (\AK) } s^\top \nabla h(w^{K, t-1}).
\end{align}
Since this is a linear objective, the minimizer is attained at an extreme point of $\conv(\AK)$, and hence the search space can be reduced to $\AK$. Finally, instead of computing a convex combination of $s^t$ and $w^{K, t-1}$ as in the classical Frank-Wolfe algorithm, the method optimizes over the convex hull of the new active set, namely
\begin{align}
        \mathcal{S}^{t} &= \mathcal{S}^{t-1} \cup \{s^t\} \nonumber \\
        \label{eq:FCFW-convex-hull-optimization}
        w^{ t} &= \argmin_{w \in \conv \mathcal{S}^{t}} h(w)
\end{align}
Depending on the structure of the objective function, step~\eqref{eq:FCFW-convex-hull-optimization} can be intractable or too costly to compute, which limits the use of FCFW in some practical applications. However, since $h$ is a quadratic here, \eqref{eq:FCFW-convex-hull-optimization} can be rewritten as
\begin{align}
    \label{eq:FCFW-QP}
    \argmin_{\beta \in \Delta_t}\ &\frac{1}{2} \norm{\sum_{l=0}^t \beta_l s^l - \frac{w^K}{n}}^2.
\end{align}
This a quadratic programming (QP) problem, which can be solved using classical off-the-shelf QP solvers~\cite{osqp,aps2019mosek,ferreau2014qpoases,bambade2023proxqp}, or via first-order methods like gradient descent, since projection on the simplex can be efficiently computed. We summarize the full method in~\Cref{alg:approximate-Carath-2}.

\begin{algorithm}
\caption{Approximate Carath\'eodory via fully corrective Frank Wolfe}
\label{alg:approximate-Carath-2}
\begin{algorithmic}[1]
\STATE{ Input: $w^K/n$, set $\AK$, number of iterations $T$.}
\STATE{Pick some $w \in \AK$ and set $w^{K, 0} = w$, $s^0 = w^{K, 0}$ and $\mathcal{S}^0 = \{s^0\}$.}
\FOR{$t=1, \dots, T$}
\STATE{Find $s^t$ such that
\begin{align*}
    s^t \in \argmin_{s \in \AK} s^\top (w^{ K, t-1} - w^K/n).
\end{align*}
}
\STATE{Set $\mathcal{S}^t = \mathcal{S}^{t-1} \cup \{s^t\}$.}
\STATE{Build $S_t \in \R^{p \times (t+1)}$, the matrix whose columns are the vectors in $\mathcal{S}^t$.}
\STATE{Solve the following QP:
\begin{align*}
    \beta^t \in \argmin_{\beta \geq 0, \ones^\top \beta = 1}\ &\frac{1}{2} \beta^\top S_t^\top S_t \beta - \left(\frac{w^K}{n}\right)^\top S_t \beta
\end{align*}
}
\STATE{$w^{K, t} = \sum_{l=0}^t \beta^t_l s^l$.}
\ENDFOR
\RETURN $w^{K,T}$
\end{algorithmic}
\end{algorithm}

We now give convergence guarantees on the distance between $w^{K, T}$ and $w^K/n$ for~\Cref{alg:approximate-Carath-2}. Linear convergence rates for conditional gradient methods usually depend either on the distance from the minimizer to the relative boundary of the feasible set~\cite{guelat1986some,garber2015faster}, or on constants quantifying the conditioning of the feasible set, such as the \textit{pyramidal width}~\cite{lacoste2015global}, the \textit{vertex-facet distance}~\cite{beck2017linearly}, or the \textit{restricted width}~\cite{pena2016neumann}. Those constants are notoriously difficult to apprehend, let alone compute in practical settings. 

We give two convergence results, one based on the distance $r$ between $w^K/n$ and the relative boundary of $\conv \AK$~\cite{garber2015faster}, and the second one based on the notion of \textit{facial distance} of the feasible set~\cite{pena2019polytope}, which turns out to be equivalent to the pyramidal width defined in~\cite{lacoste2015global}.

\begin{definition}
    The facial distance of $\AK$ is defined as
    \begin{align}
    \label{eq:def-facial-distance}
        \facialdistance_{\AK} := \min_{\substack{F \in \text{faces}(\conv \AK), \\ \emptyset \subsetneq F \subsetneq \conv \AK}} \text{dist}(F, \conv (\AK \setminus F))
    \end{align}
\end{definition}

We summarize the two convergence results in the next theorem. 

\begin{prop}
    \label{prop:approximate-carath-bound}
    Let $w^{K, T}$ be the output of~\Cref{alg:approximate-Carath-2} after $T$ iterations. Then
    \begin{align}
    \label{eq:rate-approx-carath-facial-distance}
        \norm{w^{K,  T} - \frac{w^K}{n}}^2 \leq \norm{w^{ K, 0} - \frac{w^K}{n}}^2 \left(1 - \frac{ \FWLinearRateConstant(\AK, w^*)}{4\diam_{\AK}^2} \right)^{T},
    \end{align}
    where $\diam_{\AK}$ is the diameter of $\conv(\AK)$ and
\begin{align*}
\FWLinearRateConstant(\AK, w^*) := \max\left( \facialdistance_{\AK}^2, (r^*)^2\right),
    \end{align*} 
    where $r^*$ is the radius of the largest relative ball centered at $\frac{w^K}{n}$ contained in $\conv(\AK)$, and $\facialdistance_{\AK}$ is the facial distance of $\AK$ defined in~\eqref{eq:def-facial-distance}.
\end{prop}
\begin{proof}
    Observe that the objective function $h$ is $1$-strongly convex and its gradient is $1$-Lipschitz, and its optimal value is 0.  The rate with respect to the facial distance is then a consequence of~\cite[Theorem 1]{lacoste2015global} and the fact that the facial distance is equivalent to the pyramidal width~\cite[Theorem 2]{pena2019polytope}. The rate with respect to the distance to the relative boundary is a consequence of~\cite[Section 4.2]{garber2015faster}, and was also derived in~\cite{combettes2023revisiting}.
\end{proof}

\subsection{Application to trimming the solution of the first stage}
 The output $w^{K, T}$ of \Cref{alg:approximate-Carath-2} is a convex combination of vectors in $\AK$, i.e. we can write it as 
\begin{align}
\label{eq:wtilde-definition}
    w^{K, T} = \sum_{l=0}^{T} \beta_l \begin{pmatrix} 
            f_{i_l}(y_{i_l}^{k_l}) \\ A_{i_l} y_{i_l}^{k_l} \\ e_{i_l}
        \end{pmatrix} = \sum_{i=1}^n \sum_{l \in I_i} \beta_l \begin{pmatrix}
        f_i(y_i^{k_l}) \\ A_i y_i^{k_l} \\ e_i
    \end{pmatrix},
\end{align}
where, as in~\eqref{eq:trimming-result-regrouping}, we regrouped the terms in $i$ by defining $I_i := \{ l \mid i_l = i\}$ for each $i=1, \dots, n$. For each $l$, define $\alpha_l = \frac{\beta_l}{\sum_{l' \in I_{i_l}} \beta_{l'}}$ and set
\begin{align}
\label{eq:wKT-definition}
    \bar{w}^{K, T} = \sum_{i=1}^n \sum_{l \in I_i} \alpha_l \begin{pmatrix}
        f_i(y_i^{k_l}) \\ A_i y_i^{k_l} \\ e_i
    \end{pmatrix}.
\end{align}
The next theorem gives a bound on the quality of the approximation of $w^K$ by $\bar{w}^{K, T}$, and provides conditions under which $\bar{w}^{K, T}$ is well-defined. For clarity, and because it will be useful later on, we separate the first coordinate (the one corresponding to function values) from the coordinates corresponding to matrix vector multiplication.

\begin{restatable}{lem}{ApproxCarathSpecified}
\label{lem:approximate-carath-bound-specified}
    Suppose we run \Cref{alg:Constructive-caratheodory} for $T \geq n - 1$ iterations with input $\AK$ and vector $w^K/n$, which outputs $w^{K, T}$ defined in~\eqref{eq:wtilde-definition}, and that we build $\bar{w}^{K, T}$ as in~\eqref{eq:wKT-definition}. 
    Let
    \begin{align*}
        h_0 = \norm{w^{K, 0} - \frac{w^K}{n}},
    \end{align*}
    where $w^{K, 0}$ is the starting point of \Cref{alg:approximate-Carath-2}. Let $\diamAk$ be the diameter of $\AK$ and let $\FWLinearRateConstant(\AK, \frac{w^K}{n})$ be as defined in Proposition~\ref{prop:approximate-carath-bound}. Assume that $T$ is large enough so that
\begin{align}
\label{eq:necessary-condition-approx-carath-specified}
        \frac{1}{n} > h_0 \exp \left(- \frac{T}{8} \left(\frac{\FWLinearRateConstant(\AK, \frac{w^K}{n})}{\diamAk^2}\right) \right).
    \end{align}
    It then holds that $|I_i| \geq 1$ and that $\sum_{l \in I_i} \alpha_l = 1$ for all $i$. Importantly here, $|I_i| > 1$ for at most $T - n + 1$ indices $i$. Moreover,
    \begin{align*}
        \abs{\sum_{i=1}^n \sum_{l \in I_i} \alpha_l f_i(y_i^{k_l}) - w^K_1} &\leq n \left( \sqrt{n} \max_{i, y \in \domfi}\abs{
        f_i(y) } + 1\right) h_0 \exp\left(-\frac{T  \FWLinearRateConstant(\AK, \frac{w^K}{n})}{8 \diamAk^2}\right), \\
            \norm{\sum_{i=1}^n \sum_{l \in I_i} \alpha_l A_iy_i^{k_l} - w^K_{2:m+1}} &\leq n \left( \sqrt{n} \max_{i, y \in \domfi}\norm{
        A_iy } + 1\right) h_0 \exp\left(-\frac{T \FWLinearRateConstant(\AK, \frac{w^K}{n})}{8 \diamAk^2}\right).
    \end{align*}
\end{restatable}
 \begin{proof}
We start by showing that $|I_i| \geq 1$ for all $i$, which implies that $\alpha_l$ is well defined for all $l$. Indeed, the bound of Proposition~\ref{prop:approximate-carath-bound} can be applied in particular to the last $n$ coordinates of $w^{K, T} - \frac{w^K}{n}$. This gives, for all $i=1, \dots, n$,
\begin{align*}
        \left( \sum_{l\in I_i} \beta_l - \frac{1}{n} \right)^2 \leq h_0^2 \left(1 - \frac{1}{4} \left(\frac{\FWLinearRateConstant(\AK, \frac{w^K}{n})}{\diamAk^2}\right) \right)^{T} \leq h_0^2 \exp\left(- \frac{T}{4} \left(\frac{\FWLinearRateConstant(\AK, \frac{w^K}{n})}{\diamAk^2}\right) \right).
    \end{align*}
    If, for some $i$, $I_i$ is empty, this implies that
    \begin{align*}
        \frac{1}{n} \leq h_0 \exp \left(- \frac{T}{8} \left(\frac{\FWLinearRateConstant(\AK, \frac{w^K}{n})}{\diamAk^2}\right) \right),
    \end{align*}
    which contradicts our assumption. Therefore $|I_i| \geq 1$ for all $i=1, \dots, n$. Moreover, the definition of $w^{K, T}$ tells us that there are only $T+1$ terms in the sum, which implies that
    \begin{align*}
        T+1 = \sum_{i=1}^n |I_i|.
    \end{align*}
    Since $|I_i| \geq 1$ for all $i$, this implies that at $|I_i| > 1$ for at most $T+1 - n$ indices.
    
    We now prove the inequalities. We start with the first one. The second one follows with the same argument. By the triangle inequality,
    \begin{align*}
        \abs{\bar{w}^{K, T} - w^K_1}
            &\leq  \abs{\sum_{i=1}^n \sum_{l \in I_i} \alpha_l f_i(y_i^{i_l}) - nw^{K, T}_1}+ n \abs{w^{K, T}_1 - \frac{w^K_1}{n}} \\
            &= \abs{\sum_{i=1}^n \sum_{l \in I_i} \alpha_l f_i(y_i^{k_l}) - n\sum_{i=1}^n \sum_{l \in I_i} \beta_l f_i(y_i^{k_l})} + n \abs{w^{K, T}_1 - \frac{w^K_1}{n}}
    \end{align*}
    The second term in the sum can be bounded using the rate from Proposition~\ref{prop:approximate-carath-bound}, which gives
    \begin{align*}
       n \abs{w^{K, T}_1 - \frac{w^K_1}{n}}  \leq n h_0 \exp \left( -\frac{T\FWLinearRateConstant(\AK, \frac{w^K}{n})}{8\diamAk^2} \right).
    \end{align*}
    For the first term, we have 
    \begin{align*}
        \abs{\sum_{i=1}^n \sum_{l \in I_i} \alpha_l f_i(y_i^{k_l}) - n\sum_{i=1}^n \sum_{l \in I_i} \beta_l f_i(y_i^{k_l})} &= \abs{\sum_{i=1}^n \left(\sum_{l \in I_i} \alpha_l 
        f_i(y_{i}^{k_l})  - \sum_{l \in I_i} n \beta_l
        f_i(y_{i}^{k_l}) \right)} \\
    &= \abs{\sum_{i=1}^n \sum_{l \in I_i} (\alpha_l - n\beta_l) 
        f_i(y_{i}^{k_l})}\\
    &\leq \max_{i, y \in \domfi}\abs{
        f_i(y) }  \sum_{i=1}^n \sum_{l \in I_i} \abs{\alpha_l - n \beta_l} \\
    &= \max_{i, y \in \domfi}\abs{
        f_i(y) }  \sum_{i=1}^n \sum_{l \in I_i} \abs{\beta_l \left(\frac{1}{\sum_{l' \in I_i} \beta_{l'}} - n \right)} \\
     &= \max_{i, y \in \domfi}\abs{
        f_i(y) } \sum_{i=1}^n \sum_{l \in I_i} \beta_l \abs{ \frac{1}{\sum_{l' \in I_i} \beta_{l'}} - n }
        \end{align*}
Now, since $\sum_{l\in I_i } \beta_l > 0$ and the term in absolute value does not depend on $l \in I_i$,
\begin{align*}
    \abs{\sum_{i=1}^n \sum_{l \in I_i} \alpha_l f_i(y_i^{i_l}) - n\sum_{i=1}^n \sum_{l \in I_i} \beta_l f_i(y_i^{i_l})} &\leq \max_{i, y \in \domfi}\abs{
        f_i(y) }  \sum_{i=1}^n \abs{ \sum_{l \in I_i} \beta_l  \left(\frac{1}{\sum_{l' \in I_i} \beta_{l'}} - n \right)} \\
        &= \max_{i, y \in \domfi}\abs{
        f_i(y) } \sum_{i=1}^n n \abs{  \frac{1}{n} - \sum_{l \in I_i} \beta_l}\\
    &\leq \max_{i, y \in \domfi}\abs{
        f_i(y) } n\sqrt{n} \sqrt{\sum_{i=1}^n  \left(  \frac{1}{n} - \sum_{l \in I_i} \beta_l \right)^2}\\
    &\leq n\sqrt{n}\max_{i, y \in \domfi}\abs{
        f_i(y) } h_0 \exp\left( -\frac{T\FWLinearRateConstant(\AK, \frac{w^K}{n})}{8 \diamAk^2} \right).
    \end{align*} 
    The second to last inequality comes from the fact that $\norm{x}_1 \leq \sqrt{n} \norm{x}_2$ for any $x \in \R^n$, while the last one comes from the rate given in Proposition~\ref{prop:approximate-carath-bound}.
\end{proof}

\begin{remark}
\label{rem:sufficient-condition-approx-carath}
One needs $|I_i| \geq 1$ for all $i$ for $\bar{w}^{K, T}$ to be well-defined. Inequality~\eqref{eq:necessary-condition-approx-carath-specified} provides a sufficient condition for it to be true. It is however hard to check in practice, since both $\FWLinearRateConstant(\AK, \frac{w^K}{n})$ and $\diamAk$ are usually not available. Simple algebraic manipulations show that a sufficient condition for it to hold is
    \begin{align*}
        n \geq \frac{8 \diamAk }{\FWLinearRateConstant(\AK, \frac{w^K}{n})} \ln (nh_0).
    \end{align*}
    Since both $\diamAk$ and $\FWLinearRateConstant(\AK, \frac{w^K}{n})$ depend on $n$, we cannot be sure that this holds, even as $n$ grows. In the considered practical application, we observed that one had $|I_i| \geq 1$ for all $i$ for values of $T$ smaller than $n+m$.
\end{remark}

\subsection{Final reconstruction}
From the decomposition of $\bar{w}^{K, T}$ as defined in~\eqref{eq:wKT-definition}, we are able to build a final solution to the original problem, just as in \Cref{sec:final-reconstruction}. We can moreover quantify the quality of the resulting solution by using the bound on the approximation of $w^K$ by $\bar{w}^{K, T}$ in Lemma~\ref{lem:approximate-carath-bound-specified}. We first state the analogous of Lemma~\ref{lem:reconstruction-technical-lemma} in the case of approximate Carath\'eodory representations.

In order to lighten notations, we define the error terms induced by trimming using \Cref{alg:approximate-Carath-2} as
 \begin{align*}
     \Errf(n, T) &:= n \left( \sqrt{n} \max_{i, y \in \domfi}\abs{
        f_i(y) } + 1\right) h_0 \exp\left(-\frac{T\FWLinearRateConstant(\AK, \frac{w^K}{n})}{8 \diamAk^2}\right) \\
        \ErrA(n, T) &:= n \left( \sqrt{n} \max_{i, y \in \domfi}\norm{ A_i y
    } + 1\right) h_0 \exp\left(-\frac{T \FWLinearRateConstant(\AK, \frac{w^K}{n})}{8 \diamAk^2}\right),
 \end{align*}
 where $h_0$, $\FWLinearRateConstant(\AK, \frac{w^K}{n})$ and $\diamAk$ are defined in Lemma~\ref{lem:approximate-carath-bound-specified}.

 \begin{lem}
 \label{lem:reconstruction-technical-lemma-approx-carat}
    Let $\vstar$ be the optimal value of~\eqref{eq:dual} and~\eqref{eq:bi-dual}. Assume that we run~\Cref{alg:FW1} for $K$ iterations, followed by \Cref{alg:approximate-Carath-2} for $T \geq n - 1$ iterations to trim the number of elements, outputting some $w^{K, T}$. Finally assume that we build $\bar{w}^{K, T}$ as in~\eqref{eq:wKT-definition}. Then, under the conditions of Lemma~\ref{lem:approximate-carath-bound-specified},
    \begin{align*}
        &\bar{w}^{K, T}_1 \leq v^* + \frac{2\diamC}{\sqrt{K+1}} + \Errf(n, T),\\
        &\norm{\bar{w}^{K,T}_{2:m+1} - b}_+ \leq \frac{2\diamC}{\sqrt{K+1}} + \ErrA(n, T).
    \end{align*}
\end{lem}
\begin{proof}
    Observe that
    \begin{align*}
        \bar{w}^{K, T}_1 - \vstar &= (\bar{w}^{K, T}_1 - w^K_1) + (w^K_1 - \vstar) ,\\
        \norm{\bar{w}^{K, T}_{2:m+1} - b}_+ &\leq \norm{\bar{w}^{K, T}_{2:m+1} - w^K_{2:m+1}}_+ + \norm{w^K_{2:m+1} - b}_+.
    \end{align*}
    Since $z^K$, namely the output of \Cref{alg:FW1}, consists of the first $1+m$ coordinates of $w^K$, we have in particular $w^K_1 = z^K_1$ and $w^K_{2:m+1} = z^K_{2:m+1}$. Applying  Theorem~\ref{thm:FW1-rate} then gives the bound on the second term of both equation. We can then bound $\bar{w}^{K, T}_1 - w^K_1$ and $\norm{\bar{w}^{K, T}_{2:m+1} - w^K_{2:m+1}}_+$ using Lemma~\ref{lem:approximate-carath-bound-specified}.
\end{proof}

Based on the previous lemmas, and using the decomposition~\eqref{eq:wKT-definition}, the exact same final reconstruction schemes as in \Cref{sec:final-reconstruction} can be used. The theoretical guarantees obtained on the reconstruction are very similar, with two important differences. First, the dependence of the duality gap bound on the nonconvexity of the function is no longer governed by $m+1$, but rather by $T + 1 - n$, where $T$ is the number of iterations of the conditional gradient method~\Cref{alg:approximate-Carath-2}. This parameter $T$ can be controlled, and hence one can reduce the impact of the nonconvexity by choosing $T \leq n + m$. This comes at the price of additive terms $\Errf(n, T)$ and $\ErrA(n, T)$ in the bounds, which are exponentially decreasing with $T$. We explicitly state and prove the result for convex domains in Corollary~\ref{cor:convergence-convex-domain-approx-carath}, which is the analogue of Theorem~\ref{thm:convergence-convex-domain}. The exact same holds for the analogue of Theorem~\ref{thm:convergence-sampled-solution}, Theorem~\ref{thm:convergence-noncconvex-domain1}, Theorem~\ref{thm:convergence-noncconvex-domain2}, and we skip them in the sake of conciseness.

First let us state the following assumption, which summarizes the algorithmic framework we are working with.
\begin{assumption}
    \label{ass:algorithmic-framework-approx-carath}
    Let $\vstar$ be the optimal value of~\eqref{eq:dual} and~\eqref{eq:bi-dual}. We assume that we run~\Cref{alg:FW1} for $K$ iterations, followed by \Cref{alg:approximate-Carath-2} for $T \geq n - 1$ iterations to trim the number of elements, outputting some $w^{K, T}$. We then build $\bar{w}^{K, T}$ as in~\eqref{eq:wKT-definition}, and we assume that the condition~\eqref{eq:necessary-condition-approx-carath-specified} of Lemma~\ref{lem:approximate-carath-bound-specified} holds.
\end{assumption}

\subsubsection{Convex domains}
Recall that for convex domains, a natural candidate solution, based on the decomposition of $\bar{w}^{K, T}$ given in~\eqref{eq:wKT-definition}, is $\bar{x} \in \R^d$ defined as
\begin{align}
\label{eq:reconstruction-CC-approx-carath}
    \bar{x}_i = \sum_{l \in I_i} \alpha_l y_{i}^{k_l}, \quad i=1, \dots, n.
\end{align}

\begin{cor}[Analogue of Theorem~\ref{thm:convergence-convex-domain}]
\label{cor:convergence-convex-domain-approx-carath}
    Suppose Assumption~\ref{ass:algorithmic-framework-approx-carath} holds and that $\domfi$ is convex for all $i=1, \dots, n$. Then, $\bar{x}$, as defined in~\eqref{eq:reconstruction-CC-approx-carath}, satisfies
\begin{align*}
    \sum_{i=1}^n f_i(\bar{x}_i) &\leq \vstar +  \frac{2\diamC}{\sqrt{K+1}}+  \Errf(n, T) + (T + 1 -n) \max_{i} \rho (f_i),\\
         \norm{\sum_{i=1}^n A_i \bar{x}_i - b}_+   &\leq \frac{2\diamC}{\sqrt{K+1}} + \ErrA(n, T).
\end{align*}
\end{cor}
\begin{proof}
Observe that
\begin{align}
\begin{split}
    \sum_{i=1}^n f_i(\bar{x}_i) - \vstar &=  \left(\sum_{i=1}^n f_i(\bar{x}_i) - \bar{w}^{K, T}_1 \right)+ \left(\bar{w}^{K, T}_1 - \vstar\right)\\
    \norm{\sum_{i=1}^n A_i\bar{x}_i - b}_+ &\leq  \norm{\sum_{i=1}^n A_i \bar{x}_i - \bar{w}^{K, T}_{2:m+1}}_+ + \norm{\bar{w}^{K, T}_{2:m+1} - b}_+
    \end{split}
\end{align}
We can apply Lemma~\ref{lem:reconstruction-technical-lemma-approx-carat} to bound the last term of each equation. Moreover,
    \begin{align*}
    \sum_{i=1}^n f_i(\bar{x}_i) - \bar{w}^{K, T}_1  &= \sum_{i=1}^n f_i(\bar{x}_i) - \sum_{i=1}^n \sum_{l\in I_i} \alpha_l f_i(y_{i}^{k_l}) \\
        &= \sum_{i, |I_i| > 1}  \left( f_i\left(\sum_{l \in I_i} \alpha_l y_{i}^{k_l}\right) - \sum_{l \in I_i}\alpha_l f_i(y_{i}^{k_l}) \right)\\
        &\leq \sum_{i, |I_i| > 1} \rho(f_i) \\
        &\leq (T + 1 - n) \max_{i} \rho(f_i),
    \end{align*}
    and
    \begin{align*}
         \sum_{i=1}^n A_i \bar{x}_i
         - \bar{w}^{K, T}_{2:m+1} &=  \sum_{i=1}^n A_i \sum_{l\in I_i} \alpha_l y_{i}^{k_l}
         - \sum_{i=1}^n \sum_{l \in I_i} \alpha_l A_i y_{i}^{k_l} = 0.
    \end{align*}
\end{proof}

\begin{remark}
\label{rem:MNP}
    In this section we have explored how one can use the fully-corrective Frank-Wolfe algorithm (FCFW) to efficiently compute approximate Carath\'eodory representation of the output $w^K$ of the first stage. It is important to note that other linearly convergent variants of Frank-Wolfe may be used. In particular, away-steps~\cite{lacoste2015global} Frank-Wolfe and the min-norm point (MNP) algorithm~\cite{wolfe1976finding} are natural candidates. The away-steps Frank-Wolfe variant did not yield good numerical performance, but the MNP algorithm did, so we mention it here. Like FCFW, it is a conditional gradient method relying on the same linear minimization step~\eqref{eq:fcfw-lmo}. The main difference with FCFW is that instead of optimizing over a convex hull at each iteration (as in~\eqref{eq:FCFW-convex-hull-optimization}), it optimizes over the affine hull of an active set. This is very interesting, as this comes down to solving a linear system at each iteration, instead of solving a QP for FCFW (as in~\eqref{eq:FCFW-QP}). The linear system can be kept in triangular form and efficiently updated, which makes iterations of MNP very efficient in practice (see~\cite{wolfe1976finding} for the full algorithm description and implementation details). Moreover, MNP also exhibits a linear rate of convergence~\cite{lacoste2015global}, only worse than the rate of FCFW by a factor 2 in the exponent. More precisely, Proposition~\ref{prop:approximate-carath-bound} still holds with 
    \begin{align}
        \norm{w^{K,  T} - \frac{w^K}{n}}^2 \leq \norm{w^{ K, 0} - \frac{w^K}{n}}^2 \left(1 - \frac{ \FWLinearRateConstant(\AK, w^*)}{4\diam_{\AK}^2} \right)^{T/2}.
    \end{align}
    The remaining results of this section still hold, by simply replacing $T$ by $T/2$ in condition~\eqref{eq:necessary-condition-approx-carath-specified} and in the definitions of $\Errf(n, T)$ and $\ErrA(n, T)$.

    Finally, note that it might be possible to obtain approximate Carath\'eodory representations via other methods, such as simplicial decomposition ones~\cite{holloway1974extension}, of which the MNP algorithm is a subcase.
\end{remark}

\section{Numerical results}
\label{sec:numerical-results}
In this section we explore the numerical performance of our method on two problems, a unit commitment problem and a problem of charging a fleet of electric vehicles. The code to reproduce the experiments is available at: \url{https://github.com/bpauld/NonConvexOpt}

\subsection{Unit Commitment problem}
The unit commitment (UC) problem aims at finding a scheduling of power-generating units with minimum cost, subject to a demand constraint. We first describe the problem. For an in-depth survey of different UC problems and solvers, see~\cite{padhy2004unit,mallipeddi2014unit}. We are given a network of $n$ power-generating units, and a time-horizon $N$. At each time step $t\in \{1, \dots, N\}$, unit $i$ can either be turned on or off, which is denoted by a binary variable $u^i_{t} \in \{0, 1\}$. Turning on unit $i$ has cost $c_{01}^{i}$, and turning it off has cost $c_{10}^{i}$. The power output of unit $i$ at time step $t$ is written $g_i^t \in \R$. When $u^i_t = 0$, unit $i$ is turned off and $g^i_t = 0$. When $u^i_t = 1$, the power generation $g^i_t$ lies between a minimum value $\underline{g}^{i}$ and a maximum value $\overline{g}^i$. As is standard for UC problems~\cite{fattahi2017conic}, we assume the cost of production to be a quadratic function of $g^i_t$. We assume that all units are turned off at time step $0$. The cost functions can then be written
\begin{align}
    \label{eq:function-def-UC}
    f_i(u^i, g^i) &= \sum_{t=1}^N \ell_i(u^i_t, g^i_t), \quad i=1, \dots, n,
\end{align}
where
\begin{align*}
    \ell_i(u^i_t, g^i_t) = \begin{cases}
        c_{10}u^{i}_{t-1} &\text{ if } u^i_t = 0, \\
        c_{01}^i(1 - u^i_{t-1}) + \beta_i (g^i_t)^2 + \gamma_i g^i_t + \omega_i &\text{ if } u^i_t = 1,
    \end{cases}
\end{align*}
where $u^i_0 = 0$ by assumption. The domain of the functions is
\begin{align}
      \dom (f_i) = \left\{ (u^i, g^i) \in \{0, 1\}^N \times \R^N \middle\vert \begin{array}{l}
    (u^i_t, g^i_t) = (0, 0) \text{ or }\\
    (u^i_t, g^i_t) \in (1, [\underline{g}^i, \overline{g}^i]),
  \end{array} \text{ for all $t=1, \dots, N$}\right\}.
\end{align}
Finally, the network must meet a demand $D_t \in \R$ at each time step $t$. The optimization of production costs can then be written as
\begin{align}
\label{eq:UC-prob}
    \begin{split}
    \mbox{minimize} & \ \sum_{i=1}^n f_i(u^i, g^i)\\
    \text{subject to }& \ \sum_{i=1}^n g^i_t  \geq D_t, \quad t=1, \dots, N\\
    &(u_i, g_i) \in \dom (f_i), \quad i=1, \dots, n.
    \end{split}
\end{align}
This fits exactly withing the setting of~\eqref{eq:primal} with $m = N$.

\subsubsection{Computing the conjugate}
\label{sec:UC-conjugate}
In this section we show how to compute the conjugate of the function $f_i$ defined in~\eqref{eq:function-def-UC}. For ease of notation, we drop all indices $i$ and superscripts $i$. Suppose we wish to compute $f^*(v, p)$ for $v\in\R^N$ and $p\in \R^N$. Let us define
\begin{align*}
    \Delta_{k, 0} &= \max_{\substack{ u_1, \dots, u_k,\\ g_1, \dots, g_k, \\ u_k = 0}} \left\{ \sum_{t=1}^k \left( u_t v_t + g_t p_t \right) - \sum_{t=1}^k \ell(u_t, g_t) \right\}, \\
    \Delta_{k, 1} &= \max_{\substack{ u_1, \dots, u_k,\\ g_1, \dots, g_k, \\ u_k = 1}} \left\{ \sum_{t=1}^k \left( u_t v_t + g_t p_t \right) - \sum_{t=1}^k \ell(u_t, g_t) \right\}.
\end{align*}
One then notices that $f^*(v, p) = \max \left( \Delta_{N, 0}, \Delta_{N, 1}\right)$. Moreover, we have
\begin{align*}
    \Delta_{1, 0} &= 0, \quad
    \Delta_{1, 1} = v_1 - c_{01} + \max_{\underline{g} \leq g \leq \overline{g}} \left\{  p_1 g - \beta g^2 - \gamma g - \omega\right\},
\end{align*}
and for $k > 1$,
\begin{align*}
    \Delta_{k, 0} &= \max\left( \Delta_{k-1, 0}, \Delta_{k-1, 1} - c_{10}\right),\\
    \Delta_{k, 1} &= \max\left( \Delta_{k-1, 0} - c_{01}, \Delta_{k-1, 1}\right) + v_k + \max_{\underline{g} \leq g \leq \overline{g}} \left\{  p_k g - \beta g^2 - \gamma g - \omega\right\}.
\end{align*}
The objective in the maximization steps is a one-dimensional quadratic over a bounded closed interval, so obtaining the optimal value is straightforward. Thus the value $f^*(v, p)$ can be obtained recursively, and the complexity is $O(N)$. Moreover, the gradient can be obtained by keeping track of the arguments yielding the maximum in the recursion. 

\ifthenelse{\boolean{longversion}}{
\subsubsection{Linear minimization over the domain}
When $\alpha_k = 0$ in~\eqref{eq:FW1-LMO-2}, one needs to compute a linear minimization over $\dom f_i$. We show how to do so now. As in the computation of the conjugate above, we drop all indices $i$. Given $(v, p) \in \R^N \times \R^N$, one must solve
\begin{align*}
    \min_{(u, g) \in \dom f} \sum_{t=1}^N (u_t v_t + g_t p_t). 
\end{align*}
The structure of $\dom f$ tells us that the problem is separable, and the subproblems can be rewritten as
\begin{align*}
    \min_{u_t, g_t} \quad & u_t v_t + g_t p_t \\
    \text{subject to } & (u_t, g_t) = (0, 0) \text{ or } (u_t, g_t) \in (1, [\underline{g}, \overline{g}]).
\end{align*}
Since the objective is linear, the minimum is necessarily attained at an extreme point, i.e. at one of $(0, 0)$, $(1, \underline{g})$ or $(1, \overline{g})$.

}{}
\subsubsection{Implementation details}
We generate random instances of the UC problem. Details on the data generation process can be found in \Cref{sec:experimental-details}. To compute the optimal dual value $\vstar$, we solve the dual~\eqref{eq:dual} using projected gradient descent with a decreasing stepsize, which we stop when no more significant progress on the dual objective value is achieved. We run \Cref{alg:FW1} for $K=10^4$ iterations. We then either compute an exact Carath\'eodory representation (\Cref{alg:Constructive-caratheodory}), an approximate Carath\'eodory representation with FCFW (\Cref{alg:approximate-Carath-2}), or an approximate Carath\'eodory representation with MNP (see \Cref{rem:MNP} and~\cite{wolfe1976finding}). For \Cref{alg:approximate-Carath-2}, we solve the QP using ProxQP~\cite{bambade2023proxqp}.

As we discussed in \Cref{rem:comments-on-assumption}, Assumption~\ref{ass:nonconvex-domain1} holds for the unit commitment problem, namely for a point in the convex hull of the domain there always exist a a point in the domain which increases the overall electricity output. We write this solution $x^F$.

Finally, the bounds provided in \Cref{sec:final-reconstruction} give solutions that are feasible up to an arbitrarily small term. Early experiments showed that solving a slightly perturbed version of the problem yielded fully feasible primal solutions. In the spirit of the approach described in paragraph~\ref{sec:final-reconstruction-perturbed}, we thus apply our method to problem~\eqref{eq:UC-prob} perturbed by a vector $\perturb\in \R^N$. Although the discussion in~\ref{sec:final-reconstruction-perturbed} suggests setting $\perturb$ to  $(N+1)\max_i \overline{g}^i$, this turned out to be unnecessarily large in our experiments. Instead, we set $\perturb = \zeta \max_i \overline{g}^i$, with initial value $\zeta = 1$. If the solution satisfies the demand, we stop there. Otherwise, we increase $\zeta$ by 1 and solve the new perturbed problem. We repeat this until the obtained solution meets the demand.

\subsubsection{Results}
\Cref{table:UC-results} presents the results for 10 randomly generated instances with $n=50$ units over $N=10$ time steps. While the theoretical duality gap is $(N+1) \max_i \gamma(f_i)$, all of our experiments yield solutions whose duality gap is smaller than $\max_i \gamma(f_i)$, a considerable improvement. Finally, the quality of the solutions is not significantly impacted by the Carath\'eodory approach (exact or approximate).


\begin{table}
\centering
\caption{Results for different random instances of the unit commitment problem with $n=50$ and $m=10$. $x^{F}$ corresponds to the reconstructed solution whose feasibility is better than the obtained convex combination, as in \Cref{eq:reconstruction-nonconvex1}. Exact Carath\'eodory corresponds to the case where the trimming is performed using \Cref{alg:Constructive-caratheodory}, while approximate Carath\'edodory representations are computed using either FCFW (\Cref{alg:approximate-Carath-2}) or MNP (as mentionned in \Cref{rem:MNP}). We see that the recovered solutions have duality gap much below the theoretical value of $(m+1) \max_i \gamma(f_i)$.}\label{table:UC-results}
\begin{tabular}{c c c c c c c c}
\hline \hline & & \multicolumn{2}{c}{Exact Carath.} & \multicolumn{2}{c}{App. Carath. FCFW} & \multicolumn{2}{c}{App. Carath. MNP} \\
\cmidrule(lr){3-4} \cmidrule(l){5-6}  \cmidrule(l){7-8}  
$\vstar$  & $\max_i \gamma(f_i)$ & $\zeta$ & $f(x^F)$ & $\zeta$ & $f(x^F)$ & $\zeta$ & $f(x^F)$ \\\hline
87 772 &  28 395 & 1 & 96 509 & 2 & 99 924 & 1 & 95 798 \\
62 412 &  23 881 & 2 & 78 882& 2 & 74 683 & 2 & 78 588\\
62 259 &  25 079 & 1 & 70 102 & 2 & 74 318 & 1 & 71 302\\
85 452 &  23 231 & 1 & 95 037 & 3 & 106 412 & 1 & 94 137\\
113 128 & 25 301 & 1 & 123 040 & 2 & 130 858 & 1 & 123 044\\
90 860  & 23 802 & 2 & 109 873 & 3 & 119 599 & 2 & 109 673\\
78 967  & 21 190 & 1 & 87 239 & 3 & 98 516 & 1 & 88 661\\
90 168  & 23 364 & 1 & 98 793 & 2 & 105 536 & 1 & 98 805\\
108 709 & 27 126 & 1 & 118 484 & 2 & 129 752 & 1 & 119 376\\
44 832 &  22 921 & 1 & 52 141 & 2 & 53 397 & 1 & 51 271\\\hline
\end{tabular}
\end{table}

We now analyze the running time of our method. To do so, we generate random instances of the UC problem with $N = 20$ and different values of $n$. We then run \Cref{alg:FW1} for $K=10^4$. 

\Cref{fig:UC-runtime} (left) plots the average runtime of \Cref{alg:FW1} (over 5 runs). We see that there is a linear relationship between the number of units $n$ and the running time. This is expected, as the bottleneck of \Cref{alg:FW1} is the computation of the linear minimization oracle at each point $i \in \{1, \dots, n\}$. Although the running time may seem high at first, there are important things to consider here. First, the computation of the conjugate described in \Cref{sec:UC-conjugate} is an iterative approach, implemented with for loops in our code, which are notoriously slow in Python. One could thus highly benefit from an implementation in another programming language. Moreover, we implemented our method on a single machine but the structure of the linear minimization step~\eqref{eq:FW1-LMO} make it fully parallelizable, so one could further take advantage of having several machines solving~\eqref{eq:FW1-LMO-2} in parallel.

\Cref{fig:UC-runtime} (middle and right) plots the average running time of the different Carath\'eodory algorithms. We see on the middle plot that the running time of the exact algorithm (\Cref{alg:Constructive-caratheodory}) becomes prohibitively large even for moderate $n$, and hence we did not run it for values over $n=200$. The right plot shows that the running time of the MNP algorithm is better than the one of FCFW. This is expected, as each iteration of of FCFW solves a QP whose size increases with the number of iterations, and one must run at least $n$ iterations in order to construct a final solution. This plot and the result from \Cref{table:UC-results} lead us to conclude that the best method to obtain Carath\'eodory representations is the MNP algorithm.

\begin{figure}[!ht]
\centering
 \includegraphics[scale=0.4]{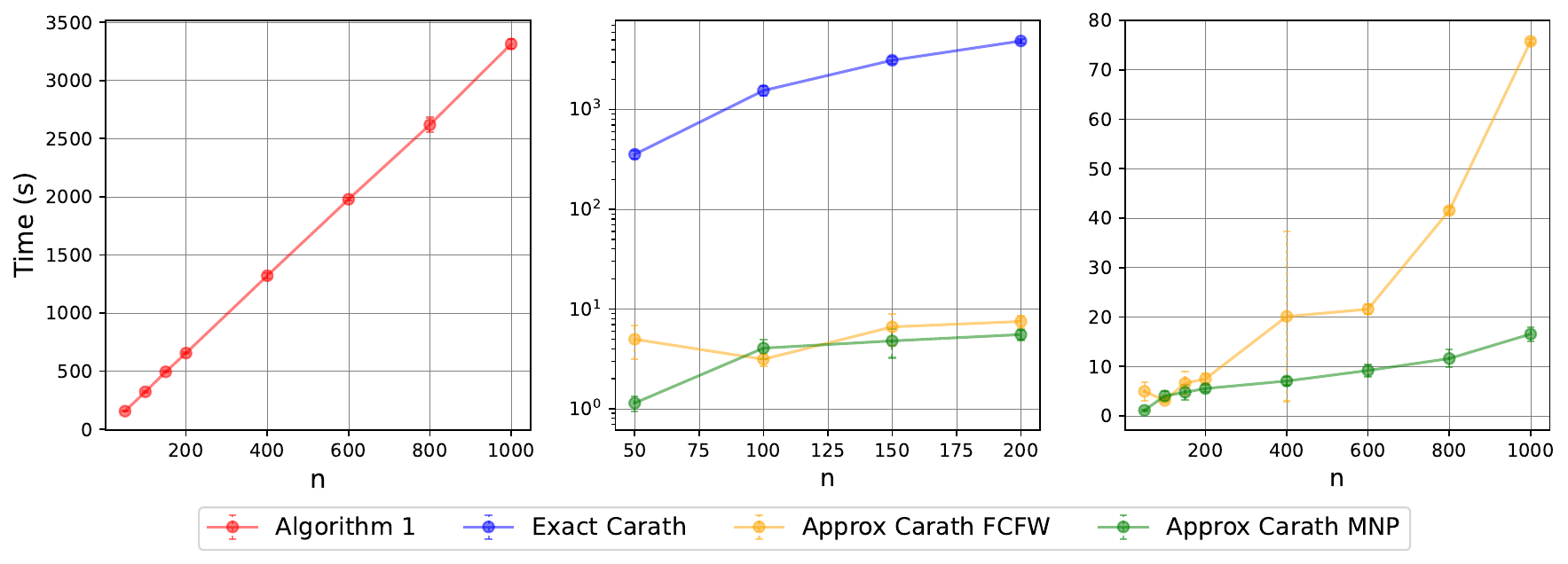}
\caption{Running time of the different algorithms with respect to $n$, averaged over 5 runs. (Left) Running time of \Cref{alg:FW1}. (Middle) Running time of all three Carath\'eodory approaches. We see that the exact approach (\Cref{alg:Constructive-caratheodory}) has long running time even for small values of $n$. (Right) Running time of the approximate Carath\'eodory approaches for larger values of $n$. We see that as $n$ grows, the MNP algorithm becomes more efficient than the FCFW algorithm.}
  \label{fig:UC-runtime}
\end{figure}


\subsection{Charging of Plug-in Electric Vehicles}

Inspired by~\cite{vujanic2016decomposition}, we now consider the problem of efficiently charging a fleet of Plug-in Electric Vehicles (PEVs). Let us describe the problem setting. We are given a fleet of $n$ vehicles, which must be charged over $N$ intervals of time, each of length $\Delta$. The charge level of vehicle $i$ at time step $k$ is denoted as $e_i^{k}$. Whether or not vehicle $i$ is charging at time step $k$ is given by a variable $u^i_k \in \{0, 1\}$. The charge level at time step $k+1$ is then given by
\begin{align*}
    e^i_{k+1} = e^i_k + P_i \Delta \xi_i u^i_k,
\end{align*}
where $P_i$ is the charging rate at station $i$ and $\xi_i$ is the charging conversion efficiency. The initial charge level is given by $E_{\text{init}}^i$, and the final charge must be at least $E_{\text{ref}}^i$ for all $i$. Moreover, the level of charging of vehicle $i$ cannot exceed a maximum value $E_{\text{max}}^i$. Finally, the system cannot impose too much stress on the overall network. This translates into the constraint
\begin{align*}
    \sum_{i=1}^n P_i u^i_k \leq P^{\text{max}}_k, \ k=1, \dots, N,
\end{align*}
for some given $P^{\text{max}} \in \R^N$. If $C^u \in \R^N$ denotes the price vector for electricity consumption at each time step, we can write the problem of optimizing costs as
\begin{align*}
    \mbox{minimize }  & \ \sum_{i=1}^n f_i(u^i)\\
    \text{subject to } & \ \sum_{i=1}^n P_i u^i_k \leq P^{\text{max}}_k, \quad k=1, \dots, N,\\
    &u^i \in \dom (f_i), \quad i=1, \dots, n,
    \end{align*}
where
\begin{align*}
    f_i(u^i) = \sum_{k=1}^{N} P_i C^u_k u^i_k
\end{align*}
and
\begin{align*}
    \dom f_i = \left\{ u \in \{0, 1\}^N\ \middle\vert \begin{array}{l}
    E_{\text{init}}^i + \sum_{k=1}^{N}P_i \Delta \xi_i u_k \geq E_{\text{ref}}^i, \\
    E_{\text{init}}^i + \sum_{k=1}^{j}P_i \Delta \xi_i u_k \leq E_{\text{max}}^i, \ j=1, \dots, N
  \end{array}\right\}
\end{align*}
The computation of the Fenchel conjugate of $f_i$ simply consists of a greedy strategy~\cite{vujanic2016decomposition}, so we omit it here.

\subsubsection{Data generation} We set $n=500$ and $N=24$ and refer the reader to~\cite[Appendix B]{vujanic2016decomposition} for the full data generation procedure.

\subsubsection{Implementation details} We implement our method with the MNP algorithm to trim the number of elements in the output of \Cref{alg:FW1}. We compare it with the randomized algorithm from~\cite{udell2016bounding}, the only other method we are aware of achieving similar duality gap bounds as our method. Since the domain of the function is not convex, we resort to the perturbation approach described in \ref{sec:final-reconstruction-perturbed} and hence apply our approach to the perturbed problem~\eqref{eq:bi-dual-perturbed}. Setting $\perturb = N \max_i P_i$ ensures that the assumptions of \ref{sec:final-reconstruction-perturbed} are satisfied~\cite{vujanic2016decomposition}. We build the final solution as in~\eqref{eq:reconstruction-max}. We also do this for the randomized algorithm from~\cite{udell2016bounding}.

\subsubsection{Results} We explore how the feasibility and the function values evolve as a function of $K$, the number of iterations of \Cref{alg:FW1}. We write $\bar{x}_K$ as the final reconstructed solution. 

\Cref{fig:PEV-unfeas} plots the average unfeasibility (over 10 runs) of $\bar{x}_K$ in the perturbed problem~\eqref{eq:bi-dual-perturbed} (left) and in the original primal problem~\eqref{eq:primal} (right). First we see that for the original primal problem, a small value of $K = 10^3$ was sufficient for all solutions to be fully primal feasible. The left plot shows that infeasibility in the perturbed primal problem behaves similarly to the (scaled) $O(1/\sqrt{K})$ curve in red, as is expected by our theoretical bounds. We hypothesize that the slightly worse than $O(1/\sqrt{K})$ behavior is due to the additive exponential error term $\ErrA(n, T)$ which appears in the bounds when using approximate Carath\'eodory algorithms (c.f. Lemma~\ref{lem:reconstruction-technical-lemma-approx-carat}).

\begin{figure}[!ht]
\centering
 \includegraphics[scale=0.4]{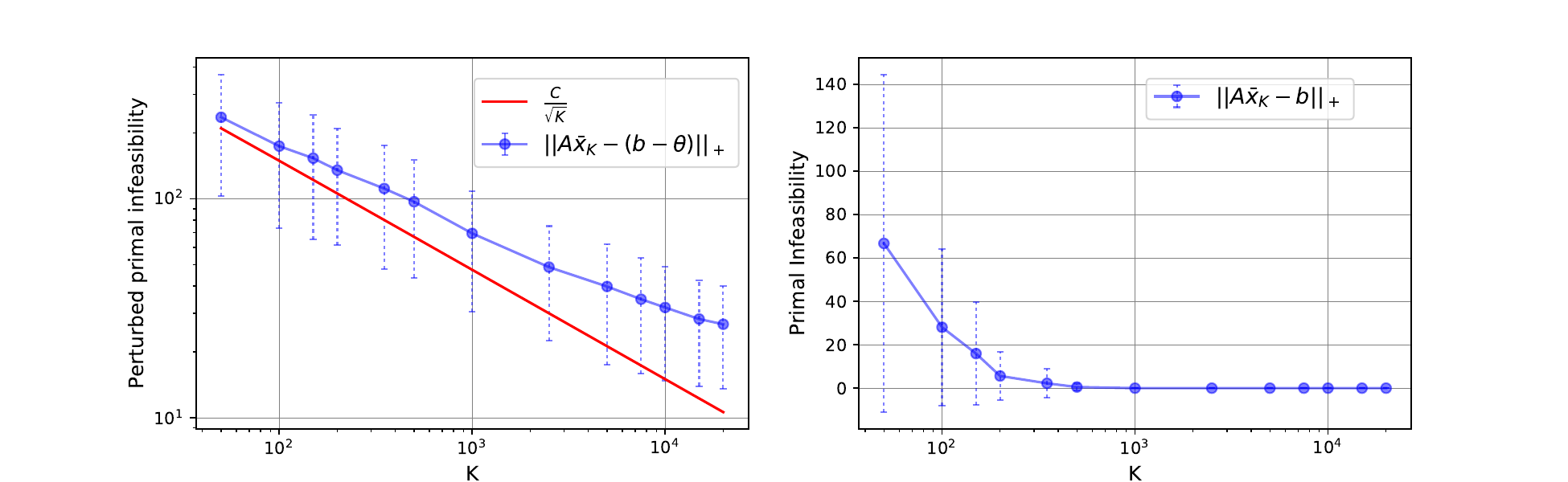}
\caption{Infeasibility of the final reconstructed solution $\bar{x}_K$ with respect to the number of iterations $K$ of \Cref{alg:FW1}, averaged over 10 runs. (Left) Infeasibility of the solution in the perturbed primal problem. In red is just the (scaled) curve $1/\sqrt{K}$ for reference. We see that the infeasibility seems to be decreasing close to the $O(1/\sqrt{K})$ rate as predicted by our bounds, and the slightly worse behavior is most likely due to the additive exponential error term $\ErrA(n, T)$. (Right) Infeasibility of the solution in the original primal problem. We see that over 10 runs, it was always sufficient to run for only $K=1000$ iterations to obtain a fully feasible solution.}
  \label{fig:PEV-unfeas}
\end{figure}

\Cref{fig:PEV-function-values} plots the evolution of the function values for two specific instances of the problem. In orange is the function value for the output of the randomized algorithm from~\cite{udell2016bounding}, while the curve in blue is the output of our method for different $K$. Finally, the vertical dashed red line represents the first value of $K$ for which the original primal problem is fully primal feasible for $\bar{x}_K$. At first sight, it might seem strange that the value of the function increases with $K$, but this is actually expected from the left plot in \Cref{fig:PEV-unfeas}. Indeed, for small values of $K$, the perturbed primal problem (which is the one we numerically solve) is more infeasible than for large values of $K$, and thus smaller objective function values can be obtained. One is thus tempted to take small values of $K$, but not too small so that the original primal problem is fully feasible. In that sense, the `optimal' value of $K$ is the one given by the red vertical line, i.e. the smallest value of $K$ for which the original primal problem is fully feasible.

For large-scale problems, we saw in the experiments on the unit commitment problem that the main bottleneck is the running time of the first algorithm \Cref{alg:FW1}. This suggests the following strategy for the method. While running \Cref{alg:FW1}, one should periodically stop the algorithm at some iteration $K$ with iterate $z_K$, compute the final solution $\bar{x}_K$ using the approximate Carath\'eodory algorithm, and check if that solution is feasible in the original primal problem. If it is, then one can stop the whole procedure. Otherwise, one can go back to \Cref{alg:FW1} and resume from $z_K$, and do another check later on. Since the running time of approximate Carath\'eodory algorithms is much lower than \Cref{alg:FW1}, this will not significantly impact the running time of the overall procedure.

\begin{figure}[!ht]
\centering
 \includegraphics[scale=0.4]{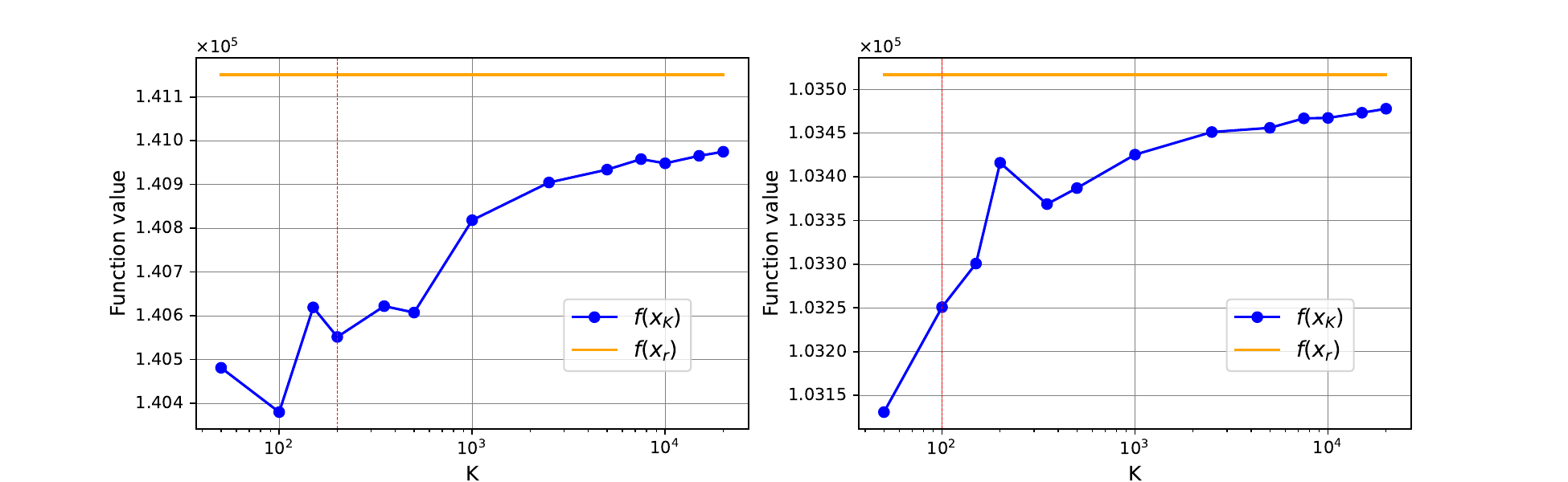}
\caption{ Typical behavior of our method on two specific instances of the PEVs problem. The curve in blue represents the value of the function at the final reconstructed solution $\bar{x}_K$ with respect to the number of iterations $K$ of \Cref{alg:FW1}. In orange is the value obtained by the randomized algorithm from~\cite{udell2016bounding}. The vertical dashed red line indicates the first value of $K$ for which $\bar{x}_K$ is primal feasible, i.e. $A \bar{x}_K - b \leq 0$.}
  \label{fig:PEV-function-values}
\end{figure}

\appendix

\section{Technical lemmas}
\label{sec:helper-lemmas}

\ExtremePointsMatching*

To prove the above lemma, we prove the following, stronger result. Recall that a function $f$ is 1-coercive if 
\begin{align*}
    \lim_{\norm{x} \rightarrow \infty} \frac{f(x)}{ \norm{x}} = \infty.
\end{align*}
In particular, a function with a compact domain is 1-coercive.
\begin{lem}
    \label{lem:extreme-points}
    Suppose $f$ is proper, closed and 1-coercive. Then
    \begin{align*}
        \left\{ (Ax, f^{**}(x)) \mid x \in \dom f^{**} \right\} \subset \conv \left\{ (Ax, f(x)) \mid x \in \dom f \right\}.
        \end{align*}
    Moreover, if for some $\tilde{x}$,
    \begin{align*}
        (A\tilde{x}, f^{**}(\tilde{x})) \in \ext \left\{ (Ax, f^{**}(x) ) \mid x \in \dom f^{**} \right\},
    \end{align*} 
    then $\tilde{x} \in \dom f$ and $f^{**}(\tilde{x}) = f(\tilde{x})$.
\end{lem}
\begin{proof}
     Let $z = (A \tilde{x}, f^{**}(\tilde{x})) \in \left\{ (Ax, f^{**}(x) ) \mid x \in \dom f^{**} \right\}$. 
     
     We have $(\tilde{x}, f^{**}(\tilde{x})) \in \epi f^{**} = \conv \epi f$ \cite[Theorem X.1.5.3]{hiriart1996convex}. Therefore, for some $K \in \N$, $\alpha \in \Delta_K$ and some $x_j \in \dom f$,
    \begin{align}
    \label{eq:helper-lemma1-convex-hull}
        (\tilde{x}, f^{**}(\tilde{x})) = \sum_{j=1}^K \alpha_i (x_j, t_j) 
        \ \text{ with }\ f(x_j) \leq t_j
        \end{align}
        and thus
        \begin{align*}
        z = (A\tilde{x}, f^{**}(\tilde{x})) = \sum_{j=1}^K \alpha_j (Ax_j, t_j)\ \text{ with }\ f(x_j) \leq t_j.
    \end{align*}
    In particular, we have
    \begin{align}
    \label{eq:helper-lemma1-nested-inequalities}
        f^{**}(\tilde{x}) = f^{**}(\sum_{j=1}^K \alpha_j x_j)
        \leq \sum_{j=1}^K \alpha_j f^{**}(x_j) 
        \leq \sum_{j=1}^K \alpha_j f(x_j) 
        \leq \sum_{j=1}^K \alpha_j t_j
        = f^{**}(\tilde{x}).
    \end{align}
    All the above inequalities are therefore equalities, and in particular,
    \begin{align}
    \label{eq:helper-lemma1-equality3}
        (A \tilde{x}, f^{**}(\tilde{x})) = \sum_{j=1}^K \alpha_j (Ax_j, f(x_j))
    \end{align}
    so that $z \in \conv \left\{ (Ax, f(x)) \mid x \in \dom f \right\}$ which proves the first statement. Up to regrouping, we can assume that $(A x_{j_1}, f^{**}(x_{j_1})) \not = (A x_{j_2}, f^{**}(x_{j_2}))$ for all $j_1\not = j_2$ in the sum. 
    
    Now suppose that $z$ is an extreme point of the set
    $\left\{ (Ax, f^{**}(x)) \mid x \in \dom f^{**} \right\}$. Recall that $z = (A \tilde{x}, f^{**}(\tilde{x}))$ and observe that~\eqref{eq:helper-lemma1-convex-hull} and~\eqref{eq:helper-lemma1-nested-inequalities} tell us that
    \begin{align*}
        z = (A \tilde{x}, f^{**}(\tilde{x})) = \sum_{j=1}^K \alpha_j (Ax_j, f^{**}(x_j)).
    \end{align*}
    Since $z$ is an extreme point, the above convex combination must be trivial, i.e. we must have $\alpha_j = 0$ for all but one index. Suppose without loss of generality that $\alpha_1 =1 $ and that $\alpha_j = 0$ for $j > 1$. We then have, by~\eqref{eq:helper-lemma1-convex-hull}, $\tilde{x} = x_1 \in \dom f$. We also have, by~\eqref{eq:helper-lemma1-equality3},  $f^{**}(\tilde{x}) = f(x_1) = f(\tilde{x})$.
\end{proof}

\begin{lem}
\label{lem:equality-epigraph}
    Suppose $f$ is proper, closed, and 1-coercive. Then
    \begin{align*}
        \left\{ \begin{pmatrix}f^{**}(x) \\ A x \end{pmatrix} \mid x \in \dom f^{**} \right\} + \R^{m+1}_+ = \conv \left\{ \begin{pmatrix}f(x) \\ A x \end{pmatrix} \mid x \in \dom f \right\} + \R^{m+1}_+
    \end{align*}
\end{lem}
\begin{proof}
Let $z \in \left\{ \begin{pmatrix}f^{**}(x) \\ A x \end{pmatrix} \mid x \in \dom f^{**} \right\} + \R^{m+1}_+$ so that 
\begin{align*}
z = \begin{pmatrix}
    f^{**}(\tilde{x}) \\ A \tilde{x}
\end{pmatrix} + \begin{pmatrix}
    \delta_1 \\ \delta_2
\end{pmatrix},
\end{align*}
with $\delta_1 \in \R_+, \delta_2 \in \R_m^+$. It is then clear that
\begin{align*}
    \begin{pmatrix}
        \tilde{x} \\ z
    \end{pmatrix} \in \left\{ (x, r_{0}, r) \mid f^{**}(x) \leq r_{0}, A x \leq r \right\}.
\end{align*}
Using~\cite[Corollary A.6]{lemarechal2001geometric}, this implies actually that
\begin{align*}
     \begin{pmatrix}
        \tilde{x} \\ z
    \end{pmatrix} \in &\cl \conv \left\{ (x, r_{0}, r) \mid f(x) \leq r_{0}, A x \leq r \right\} \\
    &= \conv \left\{ (x, r_{0}, r) \mid f(x) \leq r_{0}, A x \leq r \right\}
\end{align*}
where the equality is due to Lemma~\ref{lem:conv-is-closed}. In particular, we have
\begin{align*}
    z &\in \conv \left\{ (r_{0}, r) \mid f(x) \leq r_{0}, A x \leq r \text{ for some } x\right\}\\
    & = \conv \left\{ 
    \left\{ \begin{pmatrix}
        f(x) \\ A x
    \end{pmatrix} \mid x \in \dom f \right\} + \R^{m+1}_+ \right\} \\
    &= \conv 
    \left\{ \begin{pmatrix}
        f(x) \\ A x
    \end{pmatrix} \mid x \in \dom f \right\} + \R^{m+1}_+.
\end{align*}
Now suppose \begin{align*}
z = \begin{pmatrix}
    z_1 \\ z_2
\end{pmatrix} &\in \conv 
    \left\{ \begin{pmatrix}
        f(x) \\ A x
    \end{pmatrix} \mid x \in \dom f \right\} + \R^{m+1}_+ \\
    &= \conv \left\{ (r_{0}, r) \mid f(x) \leq r_{0}, A x \leq r \text{ for some } x\right\}.
    \end{align*}
    Then there exists $\alpha \geq 0$ such that $1^\top \alpha = 1$ and
    \begin{align*}
        (z_1, z_2) = \sum_{k} \alpha_k (f(x_k), Ax_k) + (\delta_1, \delta_2)
    \end{align*}
    for some $\delta_1 \in \R_+, \delta_2 \in \R_+^m$ and some $x_k \in \dom f$. Let us set $y = \sum_k \alpha_k x_k$. Then
    \begin{align*}
        \begin{pmatrix}
            y \\ z
        \end{pmatrix} &\in \conv \left\{ (x, r_{0}, r) \mid f(x) \leq r_{0}, A x \leq r \right\} \\
        &= \cl \conv \left\{ (x, r_{0}, r) \mid f(x) \leq r_{0}, A x \leq r \right\} \\
        &= \left\{ (x, r_{0}, r \mid f^{**}(x) \leq r_{0}, A x \leq r \right\}
    \end{align*}
    where the first equality is by Lemma~\ref{lem:conv-is-closed} and the second is again from~\cite[Corollary A.6]{lemarechal2001geometric}. This ends the proof.
\end{proof}

\begin{lem}
    \label{lem:conv-is-closed}
    Suppose $f$ is proper, closed, and 1-coercive. Then
    \begin{align*}
        \conv \left\{ (x, r_{0}, r) \mid f(x) \leq r_{0}, A x \leq r \right\}
    \end{align*}
    is closed.
\end{lem}
\begin{proof}
    Suppose $(x_k, r_{0,k}, r_k) \in  \conv \left\{ (x, r_{0}, r) \mid f(x) \leq r_{0}, A x \leq r \right\}$ converges to  $(\tilde{x}, \tilde{r}_0, \tilde{r})$. We want to show that $(\tilde{x}, \tilde{r}_0, \tilde{r}) \in \conv \left\{ (x, r_{0}, r) \mid f(x) \leq r_{0}, A x \leq r \right\} $. For all $k$, there exists $\alpha_k \geq 0$ such that $1^\top \alpha_k = 1$ and
    \begin{align*}
        (x_k, r_{0,k}, r_k) = \sum_j \alpha_{kj} (x_{kj}, r_{0,kj}, r_{kj})
    \end{align*}
    with $(x_{kj}, r_{0,kj}) \in \epi f$ and $Ax_{kj} \leq r_{kj}$. In particular, this implies that $(x_k, r_{0,k}) \in \conv \epi f$ for all $k$. Since $\conv \epi f$ is closed~\cite[Lemma X.1.5.3]{hiriart1996convex}, this implies that $(\tilde{x}, \tilde{r}_0) \in \conv \epi f$. Therefore there exists $(x_p, r_{0,p}) \in \epi f$ such that
    \begin{align*}
        (\tilde{x}, \tilde{r}_0) = \sum_{p} \beta_p (x_p, r_{0,p})
    \end{align*}
    with $\beta \geq 0$, $1^\top \beta = 1$. Moreover, we have
    \begin{align*}
        Ax_k =  \sum_{k} \alpha_{kj} A x_{kj}  \leq \sum_{k} \alpha_{kj} r_{kj} = r_k.
    \end{align*}
    Taking limits, this implies that $A\tilde{x} \leq \tilde{r}$. Define $r_p = A x_p + (\tilde{r} - A\tilde{x})$. Then $Ax_p \leq r_p$ and $\sum_{p} \beta_p r_p = A \sum_p \beta_p x_p + (\tilde{r} - A\tilde{x}) = \tilde{r}$. Thus $$(\tilde{x}, \tilde{r}_0, \tilde{r}) \in \conv \left\{ (x, r_{0}, r) \mid f(x) \leq r_{0}, A x \leq r \right\}$$.
\end{proof}

\section{Experimental details}
\label{sec:experimental-details}
In this section, we give details on the data generation of \Cref{sec:numerical-results}. We write $U(c, d)$ as the uniform distribution in the interval $[c, d]$. At each time step $t$, the demand $D_t$ is generated according to $D_t \sim U(100, 300)$. For each unit $i$, we then set the maximum and minimum generation parameters $\overline{g}^i$ and $\underline{g}^i$ as
\begin{align*}
    p^i &\sim  \frac{1}{n} \times U(100, 300)\\
    \overline{g}^i &= 2p^i\\
    \underline{g}^i &= 0.5p^i.
\end{align*}
We then set, for all units $i$,
\begin{align*}
    \beta_i &\sim U(1, 20), \quad \gamma_i \sim U(3, 5), \quad \omega_i \sim U(30, 50),\\
    c_{01}^i &=  \frac{\sum_{j} \beta_j (p^j)^2 + \gamma_j p^j + \omega_j}{2n} , \quad c_{10}^i = \frac{c_{01}^i}{4}.
\end{align*}
We found that choosing the parameters this way allows for interesting settings where the choice of when to turn on/off units of production, and of how much power should be generated, is non-trivial.

\section*{Funding}
This work was funded in part by the French government
under management of Agence Nationale de la recherche as part of the “Investissements d’avenir” program, reference ANR-19-P3IA-0001 (PRAIRIE 3IA Institute) and
a Google focused award.

%
 \section*{Conflict of interest}
 The authors declare that they have no conflict of interest.

\bibliographystyle{spmpsci}      
\bibliography{references}   

\end{document}